\documentclass[10pt,oneside,reqno]{amsart}
\usepackage{pdfcomment}
\usepackage{soul}
\usepackage[T1]{fontenc}
\usepackage{amsthm, mathtools} 

\usepackage{enumitem}
\usepackage{comment}
\usepackage[backend = biber, style = alphabetic, giveninits=true, eprint=false
]{biblatex}
\renewbibmacro{in:}{}
\renewbibmacro*{addendum+pubstate}{}	
\AtEveryBibitem{\clearlist{language}}
\AtEveryBibitem{\clearfield{urlyear}}
\AtEveryBibitem{\clearlist{pubstate}}
\DeclareBibliographyAlias{video}{misc}
\AtEveryBibitem{%
  \ifentrytype{online}
    {\clearfield{url}%
     \clearfield{urlyear}}
    {}%
}
\AtEveryBibitem{\clearfield{note}}

\usepackage{amssymb}

\usepackage{bbm} 
\usepackage{relsize}

\usepackage{geometry}

\usepackage{hyperref}
\usepackage{cleveref}
\usepackage{xurl}
\hypersetup{breaklinks=true}
\usepackage[dvipsnames]{xcolor}

\definecolor{coolblack}{rgb}{0.0, 0.18, 0.39}
\definecolor{darkslateblue}{rgb}{0.28, 0.24, 0.55}
\hypersetup{
	citecolor=darkslateblue,
        colorlinks=true,
	linkcolor=darkslateblue,
	filecolor=darkslateblue,      
	urlcolor=coolblack,
	pdftitle={Overleaf Example},
	pdfpagemode=FullScreen,
}

\usepackage{tikz}
\usetikzlibrary{arrows, matrix, decorations.pathreplacing}
\usepackage{tikz-cd}
\tikzcdset{scale cd/.style={every label/.append style={scale=#1},
    cells={nodes={scale=#1}}}}

\numberwithin{equation}{section}
\theoremstyle{definition}
\newtheorem{Definition}{Definition}[section]
\newtheorem{Example}[Definition]{Example}
\newtheorem{Remark}[Definition]{Remark}

\theoremstyle{plain}
\newtheorem{Theorem}[Definition]{Theorem}
\crefname{Theorem}{theorem}{theorems}
\Crefname{Theorem}{Theorem}{Theorems}

\newtheorem{Proposition}[Definition]{Proposition}
\crefname{Proposition}{proposition}{propositions}
\Crefname{Proposition}{Proposition}{Propositions}

\newtheorem{Corollary}[Definition]{Corollary}
\newtheorem{Lemma}[Definition]{Lemma}

\newlist{theoremparts}{enumerate}{3}
\setlist[theoremparts,1]{label=\textbf{(\alph*)},
                   ref  =\textbf{(\alph*)}}
\setlist[theoremparts,2]{label=\textbf{(\arabic*)},
                   ref  =\thetheorempartsi\textbf{(\alph*)}}
\setlist[theoremparts,3]{label=\bfseries(\roman*),
                   ref  =\thetheorempartsii\textbf{.(\roman*)}}

\crefname{theorempartsi}{part}{parts}
\crefname{theorempartsii}{part}{parts}
\crefname{theorempartsiii}{part}{parts}
\Crefname{theorempartsi}{Part}{Parts}
\Crefname{theorempartsii}{Part}{Parts}
\Crefname{theorempartsiii}{Part}{Parts}

\crefname{Corollary}{corollary}{corollaries}
\Crefname{Corollary}{Corollary}{Corollaries}

\DeclarePairedDelimiterX{\rvect}[1]{(}{)}{\,\makervect{#1}\,}

\ExplSyntaxOn
\NewDocumentCommand{\makervect}{m}
 {
  \seq_set_split:Nnn \l_tmpa_seq { , } { #1 }
  \begin{matrix}
  \seq_use:Nn \l_tmpa_seq { & }
  \end{matrix}
 }
\ExplSyntaxOff

\usepackage{caption}
\usepackage{subcaption}
\usepackage{floatrow}

\DeclareMathOperator{\Ad}{Ad}
\DeclareMathOperator{\diag}{diag}
\DeclareMathOperator{\End}{End}

\DeclareMathOperator{\wt}{wt}
\DeclareMathOperator{\U}{U}

\title{Action of $\mathfrak{osp}(1|2n)$ on polynomials tensor $\mathbb{C}^{0|2n}$}
\keywords{Lie superalgebras, orthosymplectic, oscillator representations. MSC2020: 17B10, 17B20}
\author{Dwight Anderson Williams II}
\date{\today}
\address{\parbox{\linewidth}{Department of Mathematics, Morgan State University\\1700 E Cold Spring Ln, Baltimore, MD, 21251, USA}}
\email{dwight@mathdwight.com}
\urladdr{https://mathdwight.com}

\begin{document}
        \begin{abstract}
For each positive integer $n$, the basic classical complex Lie superalgebra $\mathfrak{osp}(1|2n)$ has a unique equivalence class of infinite-dimensional completely-pointed modules, those weight modules with one-dimensional weight spaces. The polynomials $\mathbb{C}[x_{1},x_{2}, \ldots, x_{n}]$ in $n$ indeterminates is a choice representative. In the case $n > 1$, tensoring polynomials with the natural $\mathfrak{osp}(1|2n)$-module $\mathbb{C}^{1|2n}$ gives rise to a tensor product representation $V = \mathbb{C}[x_{1},x_{2}, \ldots, x_{n}] \otimes \mathbb{C}^{1|2n}$ of $\mathfrak{osp}(1|2n)$ that decomposes into two irreducible summands. These summands are understood through automorphisms of $V$ that we determine as intertwining operators describing the first summand as an isomorphic copy of $\mathbb{C}[x_{1},x_{2}, \ldots, x_{n}]$ and the second summand as $\mathbb{C}[x_{1},x_{2}, \ldots, x_{n}] \otimes \mathbb{C}^{2n}$, which does not have a natural $\mathfrak{osp}(1|2n)$-module structure and is not a paraboson Fock space with known bases. We present the intertwining operators as infinite diagonal block matrices of arrowhead matrices and give bases, along with formulas for the action of the odd root vectors (which generate $\mathfrak{osp}(1|2n)$) on these bases, for each of these conjugated oscillator realizations. We also revisit an expected difference: The decomposition of $\mathbb{C}[x] \otimes \mathbb{C}^{1|2}$ yields three irreducible summands instead of two.
    \end{abstract}
	
	\maketitle
	
	\section{Introduction}\label{sec:Intro}
If the $n$th Weyl algebra $A_{n}$ is the codomain of a Lie algebra homomorphism $\phi\!: \mathfrak{g} \rightarrow A_{n}$ from a Lie algebra $\mathfrak{g}$, then we say $\phi$ is a canonical realization \cite[Definition 1]{havlicekCanonicalRealizationsLie1975} of $\mathfrak{g}$. Classically \cite[Section 4.6]{dixmierEnvelopingAlgebras1996}, we have a canonical realization of the symplectic Lie algebra $\mathfrak{sp}(2n)$ through anti-commutators and the resulting oscillator representation as polynomials in $n$ (commuting) variables is a commonly studied \cite{benkartModulesBoundedWeight1997,brittenModulesBoundedMultiplicities1999} object. The methods of \cite{brittenModulesBoundedMultiplicities1999} and their references (in similar classical settings) include a natural tensoring of $\mathbb{C}[x_{1},x_{2},\ldots,x_{n}]$ with the standard $\mathfrak{sp}(2n)$-module $\mathbb{C}^{2n}$ in order to understand the representation theory of symplectic Lie algebras. The superalgebra perspective extends the usual symplectic canonical realization to the orthosymplectic setting; there is a surjection \cite{fronsdalEssaysSupersymmetry1986} with domain the universal enveloping algebra $U(\mathfrak{osp}(1|2n))$ of the orthosymplectic Lie superalgebra $\mathfrak{osp}(1|2n)$, which has even part a Lie algebra isomorphic to $\mathfrak{sp}(2n)$, and codomain $A_{n}$ \cite[see also][]{kanakoglouBraidedLookGreen2007, fergusonWeightModulesOrthosymplectic2015, williamsBasesInfinitedimensionalRepresentations2020}. Likewise, the space of polynomials is a well-studied oscillator representation \cite{nishiyamaOscillatorRepresentationsOrthosymplectic1990} of $\mathfrak{osp}(1|2n)$. More could be said about the physical applications of $\mathfrak{osp}(1|2n)$ \cite{ganchevLieSuperalgebraicInterpretation1980,kanakoglouBraidedLookGreen2007}, mathematical motivation \cite{coulembierOrthosymplecticLieSupergroup2010} to study orthosymplectic Lie superalgebras generally, and the history \cite{greenGeneralizedMethodField1953} of their infinite-dimensional representations in parastatistics.  

In this paper we continue the study of tensor product representations formed through the tensor products of oscillator representations and standard modules of Lie (super)algebras. The particular problem we address is the decomposition of $\mathbb{C}[x_{1},x_{2},\ldots,x_{n}] \otimes \mathbb{C}^{1|2n}$, where the super vector space $\mathbb{C}^{1|2n}$ is the standard representation of $\mathfrak{osp}(1|2n)$ and an object in the semisimple category of finite-dimensional $\mathfrak{osp}(1|2n)$-representations. (See Theorem 4.1 and the related definition of page 1 in \cite{djokovic1976semisimplicity}.) What is new here is the determination not only of irreducible summands in a complete decomposition of $\mathbb{C}[x_{1},x_{2},\ldots,x_{n}] \otimes \mathbb{C}^{1|2n}$ but the provision of bases of the summands and actions of the $\mathfrak{osp}(1|2n)$ generators upon these bases.

Now the study of infinite-dimensional representations of $\mathfrak{osp}(1|2n)$, even without the tasks of seeking bases, is generally difficult to pin down. There is some progress in determining bases of infinite-dimensional representations when limiting the scope to paraboson Fock spaces as in \cite{bisboBasesInfiniteDimensional2022a}. We note that the previous reference considers combinatorial methods to provide bases of the polynomial paraboson Fock spaces $L_{n}(p)$, that is, infinite-dimensional representations of $\mathfrak{osp}(1|2n)$ of lowest weight $(\frac{p}{2}, \ldots, \frac{p}{2})$. In contrast, we open a new direction of determining bases for infinite-dimensional $\mathfrak{osp}(1|2n)$-representations that are not paraboson Fock spaces.

Moreover, we introduce invertible intertwining operators as the crucial component to give bases for the pair of infinite-dimensional summands in the complete decomposition of the tensor product  $\mathbb{C}[x_{1}, x_{2}, \ldots, x_{n}] \otimes \mathbb{C}^{1|2n}$ into irreducible $\mathfrak{osp}(1|2)$ modules for $n > 1$. These summands are isomorphic to realizations of $\mathfrak{osp}(1|2n)$ in $\mathbb{C}[x_{1}, x_{2}, \ldots, x_{n}] \otimes \mathbb{C}^{0|1}$ and $\mathbb{C}[x_{1}, x_{2}, \ldots, x_{n}] \otimes \mathbb{C}^{0|2n}$, respectively. While $\mathbb{C}[x_{1}, x_{2}, \ldots, x_{n}] \otimes \mathbb{C}^{1|0}$ is equivalent to the usual oscillator representation $\mathcal{H}_{n}(1) \cong L_{n}(1)$ \cite[in the notation of][]{bisboBasesInfiniteDimensional2022a}, the $\mathfrak{osp}(1|2n)$-module $\mathbb{C}[x_{1}, x_{2}, \ldots, x_{n}] \otimes \mathbb{C}^{0|2n}$ is of highest-weight $(\frac{1}{2}, \frac{1}{2}, \ldots, \frac{1}{2}, -\frac{1}{2})$ and thus not isomorphic to any paraboson Fock space $L_{n}(p)$ of order $p$. Indeed, the intertwining operators we give in this paper reveal $\mathbb{C}[x_{1}, x_{2}, \ldots, x_{n}] \otimes \mathbb{C}^{0|2n}$ as an $\mathfrak{osp}(1|2n)$-supermodule despite the right tensor factor failing to be one itself. 

In the $n=1$ case, we have an additional summand in the decomposition of $\mathbb{C}[x] \otimes \mathbb{C}^{1|2}$, for a total of three, which is consistent with the theory of $\mathfrak{osp}(1|2)$-representations; see \cite[Theorem 8]{coulembierClassTensorProduct2013} or \cite{minnaertClebschGordanCoefficientsQuantum1994}. The decomposition has been described using the diagonal reduction algebra of $\mathfrak{osp}(1|2)$ \cite{hartwigDiagonalReductionAlgebra2022} and is a specific case of Theorem 1.3 in \cite{hartwigGhostCenterRepresentations2023}. We recall in the appendix a first-principles approach on why the decomposition of $\mathbb{C}[x] \otimes \mathbb{C}^{1|2}$ involves three summands instead of two and share a basis of the space of singular vectors.

\subsection{Summary of results}
The author provides \Cref{thm:basis-and-actions} that yields two intertwining operators whose restrictions lead to simple submodules of $\mathbb{C}[x_{1},x_{2}, \ldots, x_{n}] \otimes \mathbb{C}^{1|2n}$. \Cref{thm:basis-and-actions} yields \Cref{cor:decomposition-of-V}, which is the statement of the decomposition of $\mathbb{C}[x_{1},x_{2}, \ldots, x_{n}] \otimes \mathbb{C}^{1|2n}$ as an $\mathfrak{osp}(1|2n)$-module for $n > 1$. Additionally, \Cref{thm:basis-and-actions} leads to \Cref{cor:gam-osp-action,cor:gamt-osp-action} in which formulas are given for the action of $\mathfrak{osp}(1|2n)$ on polynomials tensor the space $\mathbb{C}^{1|0}$ and the action of $\mathfrak{osp}(1|2n)$ on polynomials tensor the space $\mathbb{C}^{0|2n}$, respectively.

Additionally, in \Cref{appendix:proof} of the appendix is the proof of \Cref{prop:space-of-singular-vectors} as a correction to the statement/proof of \cite[Theorem 7.1.1]{fergusonWeightModulesOrthosymplectic2015} (and its subsequent use in \cite{williamsBasesInfinitedimensionalRepresentations2020}) concerning the dimension of the space of singular vectors in the $\mathfrak{osp}(1|2)$-module $\mathbb{C}[x] \otimes \mathbb{C}^{1|2}$. 

\begin{Proposition}\label{prop:space-of-singular-vectors}\hfill

\noindent Let $(\rho, \mathbb{C}[x_{1},x_{2},\ldots,x_{n}])$ be the completely-pointed $\mathfrak{osp}(1|2n)$-representation, and let $(\tau,\mathbb{C}^{1|2n})$ be the natural matrix realization of $\mathfrak{osp}(1|2n)$. 
The space of singular vectors in the tensor product representation $(\rho \otimes \tau, \mathbb{C}[x_{1},x_{2},\ldots,x_{n}] \otimes \mathbb{C}^{1|2n})$ is two-dimensional, except when $n=1$, in which case the dimension is three.
\end{Proposition}

\begin{Theorem}\label{thm:basis-and-actions}\hfill

\noindent Denote the $\lambda$-weight space of $\mathbb{C}[x_{1},x_{2},\ldots,x_{n}] \otimes \mathbb{C}^{1|2n}$, viewed as an $\mathfrak{osp}(1|2n)$-weight module, by $(\mathbb{C}[x_{1},x_{2},\ldots,x_{n}] \otimes \mathbb{C}^{1|2n})_{\lambda}$. Write $\{X_{\pm j} \mid 1 \leq j \leq n \}$ for a generating set of odd root vectors of $\mathfrak{osp}(1|2n)$ and $\{v_{i} \mid 0 \leq i \leq 2n\}$ for the standard basis of $\mathbb{C}^{1|2n}$, where $v_{i} = \rvect{0\cdots0\!,\!1\!,\!0\cdots0}^{\intercal}$, with $1$ in the $(i+1)^{st}$ entry.
Then there exist automorphisms $\prescript{w_{1}}{}{\Gamma}$ and $\prescript{w_{2}}{}{\Gamma}$ of $\mathbb{C}[x_{1},x_{2}, \ldots, x_{n}] \otimes \mathbb{C}^{1|2n}$ such that the following hold:

\begin{theoremparts}
\item 
The image of $\,\mathbb{C}[x_{1},x_{2}, \ldots, x_{n}] \otimes \mathbb{C}^{1|0}$ under $\prescript{w_{2}}{}{\Gamma}$ has the structure of an $\mathfrak{osp}(1|2n)$-module.\label{thm:basis-and-actions:gam-image}

\item
The image of $\,\mathbb{C}[x_{1},x_{2}, \ldots, x_{n}] \otimes \mathbb{C}^{0|2n}$ under $\prescript{w_{1}}{}{\Gamma}$ has the structure of an $\mathfrak{osp}(1|2n)$-module.\label{thm:basis-and-actions:gamt-image}

\item 
The map $\prescript{w_{1}}{}{\Gamma}_{\lambda}$ denoting the restriction of $\prescript{w_{1}}{}{\Gamma}$ to $(\mathbb{C}[x_{1},x_{2},\ldots,x_{n}] \otimes \mathbb{C}^{1|2n})_{\lambda}$ is an automorphism of $(\mathbb{C}[x_{1},x_{2},\ldots,x_{n}] \otimes \mathbb{C}^{1|2n})_{\lambda}$.  Furthermore, $\prescript{w_{1}}{}{\Gamma}_{\lambda}$ can be expressed as a block within an infinite diagonal block matrix of invertible arrowhead matrices. \label{thm:basi-and-actions:arrowhead-gam}

\item 
The map $\prescript{w_{2}}{}{\Gamma}_{\lambda}$ denoting the restriction of $\prescript{w_{2}}{}{\Gamma}$ to $(\mathbb{C}[x_{1},x_{2},\ldots,x_{n}] \otimes \mathbb{C}^{1|2n})_{\lambda}$ is an automorphism of $(\mathbb{C}[x_{1},x_{2},\ldots,x_{n}] \otimes \mathbb{C}^{1|2n})_{\lambda}$. Furthermore, $\prescript{w_{2}}{}{\Gamma}_{\lambda}$ can be expressed as a block within an infinite diagonal block matrix of invertible arrowhead matrices. \label{thm:basi-and-actions:arrowhead-gamt}
\end{theoremparts}
\end{Theorem}

\begin{Corollary}[Decomposition of polynomials tensor standard module]\label{cor:decomposition-of-V}\hfill

\noindent Fix an integer $n > 1$. The super vector space $\mathbb{C}[x_{1},x_{2}, \ldots, x_{n}] \otimes \mathbb{C}^{1|2n}$ decomposes as an $\mathfrak{osp}(1|2n)$-module equal to the direct sum $\prescript{w_{1}}{}{\Gamma}\big(\mathbb{C}[x_{1},x_{2},\ldots,x_{n}] \otimes \mathbb{C}^{0|2n}\big) \oplus \prescript{w_{2}}{}{\Gamma}\big(\mathbb{C}[x_{1},x_{2},\ldots,x_{n}] \otimes \mathbb{C}^{1|0}\big)$.
\end{Corollary}

\begin{Corollary}[Action of $\mathfrak{osp}(1|2n)$ on polynomials tensor even part of standard module]\label{cor:gam-osp-action}\hfill

\item Let $1 \leq j \leq n$. The $\mathfrak{osp}(1|2n)$-action on $\mathbb{C}[x_{1},x_{2},\ldots,x_{n}] \otimes \mathbb{C}^{1|0}$ is characterized by 
\begin{align*}
X_{j}(x_{1}^{k_{1}}x_{2}^{k_{2}} \cdots x_{n}^{k_{n}} \otimes v_{0}) &= 
-\frac{1}{\sqrt{2}}x_{1}^{k_{1}} \cdots x_{j}^{k_{j}+1} \cdots x_{n}^{k_{n}} \otimes v_{0}\\
X_{-j}(x_{1}^{k_{1}}x_{2}^{k_{2}} \cdots x_{n}^{k_{n}} \otimes v_{0}) &= 
-\frac{1}{\sqrt{2}}k_{j}x_{1}^{k_{1}} \cdots x_{j}^{k_{j}-1} \cdots x_{n}^{k_{n}} \otimes v_{0}
\end{align*} 
\end{Corollary}

\begin{Corollary}[Action of $\mathfrak{osp}(1|2n)$ on polynomials tensor odd part of standard module]\label{cor:gamt-osp-action}\hfill

\item Let $1 \leq j \leq n$. The $\mathfrak{osp}(1|2n)$-action on $\mathbb{C}[x_{1},x_{2},\ldots,x_{n}] \otimes \mathbb{C}^{0|2n}$ is characterized by 
\begin{align*}
X_{j}\Big(x_{1}^{k_{1}}x_{2}^{k_{2}} \cdots x_{n}^{k_{n}} \otimes v_{i}\Big) &= 
\frac{1}{\sqrt{2}}x_{1}^{k_{1}} \cdots x_{j}^{k_{j}+1} \cdots x_{n}^{k_{n}} \otimes v_{i} - \sqrt{2} x_{1}^{k_{1}} \cdots x_{i}^{k_{i}+1} \cdots x_{n}^{k_{n}} \otimes v_{j}\\
X_{j}(x_{1}^{k_{1}}x_{2}^{k_{2}} \cdots x_{n}^{k_{n}} \otimes v_{n+i}) &= \frac{1}{\sqrt{2}}x_{1}^{k_{1}} \cdots x_{j}^{k_{j}+1} \cdots x_{n}^{k_{n}} \otimes v_{n+i} + \sqrt{2}k_{i}x_{1}^{k_{1}} \cdots x_{i}^{k_{i}-1} \cdots x_{n}^{k_{n}} \otimes v_{j}\\
X_{-j}(x_{1}^{k_{1}}x_{2}^{k_{2}} \cdots x_{n}^{k_{n}} \otimes v_{i}) &= \frac{1}{\sqrt{2}}k_{j}x_{1}^{k_{1}} \cdots x_{j}^{k_{j}-1} \cdots x_{n}^{k_{n}} \otimes v_{i} + \sqrt{2}x_{1}^{k_{1}} \cdots x_{i}^{k_{i}+1} \cdots x_{n}^{k_{n}} \otimes v_{n+j}\\
X_{-j}(x_{1}^{k_{1}}x_{2}^{k_{2}} \cdots x_{n}^{k_{n}} \otimes v_{n+i}) &= \frac{1}{\sqrt{2}}k_{j}x_{1}^{k_{1}} \cdots x_{j}^{k_{j}-1} \cdots x_{n}^{k_{n}} \otimes v_{n+i} - \sqrt{2}k_{i}x_{1}^{k_{1}} \cdots x_{i}^{k_{i}-1} \cdots x_{n}^{k_{n}} \otimes v_{n+j}
\end{align*} 
\end{Corollary}

\begin{figure}[ht!]
\tikzset{column sep=small, ampersand replacement=\&}
\centering
\begin{tikzcd}[scale cd = 1]    
{\mathbb{C}[x_{1},x_{2},\ldots,x_{n}] \otimes \mathbb{C}^{1|0}} \arrow[rr, "\prescript{w_{2}}{}{\Gamma}",] \arrow[dd] \&                                                                                                    \& {\prescript{w_{2}}{}{\Gamma}\big(\mathbb{C}[x_{1},x_{2},\ldots,x_{n}] \otimes \mathbb{C}^{1|0}\big)} \arrow[dd, "\rho \otimes \tau"] \\
                                                                      \& {} \arrow[loop, distance=3em, in=131, out=251] \arrow[no head, loop, distance=3em, in=131, out=251] \&                                    \\
{\mathbb{C}[x_{1},x_{2},\ldots,x_{n}] \otimes \mathbb{C}^{1|0}} \arrow[rr, "\prescript{w_{2}}{}{\Gamma}"']                                          \&                                                                                                     \& {\prescript{w_{2}}{}{\Gamma}\big(\mathbb{C}[x_{1},x_{2},\ldots,x_{n}] \otimes \mathbb{C}^{1|0}\big)}                                
\end{tikzcd}{\caption{Realization of $\mathfrak{osp}(1|2n)$ in $\End(\mathbb{C}[x_{1},x_{x},\ldots,x_{n}]\otimes\mathbb{C}^{1|0})$}}
\end{figure}

\begin{figure}[ht!]
\tikzset{column sep=small, ampersand replacement=\&}
\centering
\begin{tikzcd}[scale cd = 1]
{\mathbb{C}[x_{1},x_{2},\ldots,x_{n}] \otimes \mathbb{C}^{0|2n}} \arrow[rr, "\prescript{w_{1}}{}{\Gamma}",] \arrow[dd] \&                                                                                                    \& {\prescript{w_{2}}{}{\Gamma}\big(\mathbb{C}[x_{1},x_{2},\ldots,x_{n}] \otimes \mathbb{C}^{0|2n}\big)} \arrow[dd, "\rho \otimes \tau"] \\
                                                                      \& {} \arrow[loop, distance=3em, in=131, out=251] \arrow[no head, loop, distance=3em, in=131, out=251] \&                                    \\
{\mathbb{C}[x_{1},x_{2},\ldots,x_{n}] \otimes \mathbb{C}^{0|2n}} \arrow[rr, "\prescript{w_{1}}{}{\Gamma}"']                                          \&                                                                                                     \& {\prescript{w_{1}}{}{\Gamma}\big(\mathbb{C}[x_{1},x_{2},\ldots,x_{n}] \otimes \mathbb{C}^{0|2n}\big)}                                
\end{tikzcd}{\caption{Realization of $\mathfrak{osp}(1|2n)$ in $\End(\mathbb{C}[x_{1},x_{x},\ldots,x_{n}]\otimes\mathbb{C}^{0|2n})$}}
\end{figure}

\begin{Remark}
    Arrowhead matrices are defined in \cite[for a reference]{marreroAssociatingHubdirectedMultigraphs2018}. Here we rely on a particular type of arrowhead matrix which eases recognition of nonsingularity. 
\end{Remark}
    We note that Coulembier \cite{coulembierClassTensorProduct2013} has developed general results on tensor product representations of orthosymplectic Lie superalgebras of types \cite{kacLieSuperalgebras1977} $B$, $C$, and $D$. The current paper grounds realizations of $\mathfrak{osp}(1|2n)$ in detailing a usual object in an unexpected role. In particular, Coulembier acknowledges initially overlooking the $B(0,n)$ case in the published version of \cite{coulembierClassTensorProduct2013a}. The joint endeavor of large-scale parameterization of representation classes with examples uncovered and made more concrete in formulation of bases is useful for understanding the algebraic symmetry at work and prepping for applications. 

For an outline of the paper, we recall in \Cref{sec:Background} the introductory facts of orthosymplectic Lie superalgebras of type $B(0|n) = \mathfrak{osp}(1|2n)$ and their representation theory. \Cref{sec:DecomposingTPR} contains discussion and proofs of \Cref{thm:basis-and-actions} and \Cref{cor:decomposition-of-V,cor:gam-osp-action,cor:gamt-osp-action}, beginning with a recollection of a basis of singular vectors in the tensor product representation $\mathbb{C}[x_{1}, x_{2}, \ldots, x_{n}] \otimes \mathbb{C}^{1|2n}$ for $n >1 $. The appendix has a proof of \Cref{prop:space-of-singular-vectors} and a review of algebraic notions in supermathematics.

	\section*{Acknowledgments}
	The author appreciates the encouragement and love of Margine "BigMa" Griffin.  The notation in Section \ref{sec:DecomposingTPR} is inspired by her Bid Whist talents and her joy for life: Even without the big and lil jokers, she's the one who'll make you laugh.
	\section{Background}\label{sec:Background}
	Assumed categorical and algebraic notions are found in \cite{waltonSymmetriesAlgebras2024}.
	Fix the ground field $\mathbb{C}$ throughout. Denote the two-element group $\mathbb{Z}/2\mathbb{Z}$ by $\mathbb{Z}_{2}$. We establish additional conventions: We use multi-index notation to express $\mathbb{C}[\mathbf{x}] = \mathbb{C}[x_1, x_2, \dots, x_n]$, $\mathbf{k} = (k_{1}, k_{2}, \ldots, k_{n}) \in \mathbb{Z}_{\geq 0}$, $\mathbf{x}^{\mathbf{k}} = x_{1}^{k_{1}} x_{2}^{k_{2}} \cdots x_{n}^{k_{n}}$, and  $\deg(\mathbf{x}^{\mathbf{k}}) = |\mathbf{k}| = \sum\limits_{i=1}^{n} k_{i}$.
	Let $E_{ij}$ denote the $n\times n$ elementary matrix of all zeroes except 1 in the $ij$-entry. The meaning of $\pm$ or $\mp$ is given independent of occurrences in parallel: Precisely two expressions are presented in $x \pm y \mp z$ (those are $x + y - z$ and $x - y + z$) or in $x \pm y \pm z$  (those are $x + y + z$ and $x - y - z$).
	
	Certain elements of $\mathbb{C}^{n}$ are given special notation as they will be helpful to denote roots of Lie (super)algebras: 
	$\delta_{j} = \underset{1~\text{in}~j^{th}~\text{position}}{\underbrace{(0,\ldots, 0, 1, 0, \dots, 0)}}$ and $\nu_{j} = \sum\limits_{j=1}^{n} \delta_{i}$.
    
  We view our base field $\mathbb{C}= \mathbb{C}\oplus \{0\}$ as a purely even superring with its supermodules termed super vector spaces; an even bilinear form maps each pair of elements having opposite parity to zero. See \Cref{appendix:supermath} for more on super notions.
	
    \subsection{On the Lie superalgebra \texorpdfstring{$\mathfrak{osp}(1|2n)$}{osp(1|2n)}}
	We briefly restate some key facts about $\mathfrak{osp}(1|2n)$ before discussing representation theory. The Type II basic classical Lie superalgebra $\mathfrak{osp}(1|2n)$ is the super vector space of dimension $(2n^{2}+n|2n)$, super dimension $2n^{2}-n$, which preserves an even, nondegenerate, symmetric\footnote{Or supersymmetric. See the comments of Leites in \cite{leitesOddParametersGeometry2022}.} bilinear form on $\mathbb{C}^{1|2n}$. The set \begin{align}\label{osp-basis-sym}
	\{X_{\pm\delta_{i} \pm \delta_{j}},X_{\pm 2\delta_{j}}, X_{\delta_{i} - \delta_{j}}, H_{2\delta_{j}}; X_{\pm \delta_{j}} \mid 1 \leq i \neq j \leq n \}
	\end{align} 
	forms a basis of the super vector space $\mathfrak{osp}(1|2n)$, and as a Lie superalgebra, $\mathfrak{osp}(1|2n)$ is generated by the set $\{ X_{\delta_{j}}, X_{-\delta_{j}} \}$ of odd basis elements with super Lie bracket \begin{equation} \label{rels-osp}
[[X_{\xi \delta_j}, X_{\eta \delta_k}], X_{\epsilon \delta_l}] = (\epsilon-\xi) \delta_{jl} X_{\eta \delta_k} + (\epsilon-\eta) \delta_{kl} X_{\xi \delta_j},
\end{equation}
where $\xi,\eta,\epsilon \in \{-1,1\}$, as in \cite{ganchevLieSuperalgebraicInterpretation1980}. 

	\subsubsection{Standard representation}
	For a super vector space $V = V_{\bar{0}} \oplus V_{\bar{1}}$, denote by $\mathfrak{gl}(V)$ the set of all endomorphisms on $V$ (with no regard to parity). Then $\mathfrak{gl}(V)$ is a Lie superalgebra called the general linear Lie superalgebra under the super commutator on endomorphisms.
	In the finite-dimensional case, $\mathfrak{gl}(V) \cong \mathfrak{gl}(\mathbb{C}^{m|n}) = \mathfrak{gl}(m|n)$ if and only if $\dim V = (m|n)$ (and a basis of homogeneous elements $\{ z_{1}, z_{2}, \ldots, z_{m}; z_{m+1}, \ldots, z_{m+n} \}$ is fixed). Thus for V = $\mathbb{C}^{1|2n}$, let $\mathfrak{gl}(1|2n)$ be the set of all linear transformations on $\mathbb{C}^{1|2n}$ expressed as matrices with respect to a standard basis $\{v_{0}; v_{1}, \ldots, v_{2n} \}$ of $\mathbb{C}^{1|2n}$.
	Elements of $\mathfrak{gl}(1|2n)$ are then block matrices $\begin{bmatrix} \alpha & r \\ c & A \end{bmatrix}$: Here, $\alpha$ is a scalar, $r$ is a row vector, $c$ is a column vector, and $A$ is a $2n\times2n$ square matrix. The Lie super bracket on $\mathfrak{gl}(1|2n)$ is the super commutator: 
    \begin{equation*}
	    [X,Y] = XY - (-1)^{|X||Y|}YX, \text{ for all } X,Y \in \mathfrak{gl}(1|2n).\label{supercommutator}
	\end{equation*}
    The orthosymplectic Lie superalgebra $\mathfrak{osp}(1|2n)$ is the Lie subsuperalgebra of $\mathfrak{gl}(1|2n)$ in which members subscribe to the conditions 
	$\alpha = 0$, $r = (r_{1}, \ldots, r_{n}, r_{n+1}, \ldots, r_{2n})$, $c^{\intercal} = (r_{n+1}, \ldots, r_{2n}, -r_{1}, \ldots, -r_{n})$, and $A$ is an element of the symplectic Lie algebra $\mathfrak{sp}(2n)$. 
    
    The description above yields the (faithful) standard representation of $\mathfrak{osp}(1|2n)$. 
    We will write $T_{\beta}$ for the image in the standard representation of the vector $X_{\beta}$. In long form,

\begin{equation*}
	    T_{\beta} =
	    \begin{bmatrix} 0 & r \\ c & A \end{bmatrix},
	    \end{equation*}
	   where
    \begin{align*}
        r &= (0,0),\quad A =
        \begin{bmatrix} 0 & 0\\ E_{ij} + E_{ji} & 0 \end{bmatrix}, \text{ for } \beta = -\delta_{i} - \delta_{j}\\
	    r &= (0,0),\quad A = \begin{bmatrix} 0 & 0\\ E_{jj} & 0 \end{bmatrix}, \text{ for } \beta = -2\delta_{j}\\
        r &= (0,0),\quad A = \begin{bmatrix} 0 & E_{ij} + E_{ji}\\ 0 & 0 \end{bmatrix}, \text{ for } \beta = \delta_{i} + \delta_{j}\\
	    r &= (0,0),\quad A = \begin{bmatrix} 0 & E_{jj}\\ 0 & 0 \end{bmatrix}, \text{ for } \beta = 2\delta_{j}\\
	    r &= (0,0),\quad A = \begin{bmatrix} E_{ij} & 0\\ 0 & -E_{ji} \end{bmatrix}, \text{ for }
	    \beta = \delta_{i} - \delta_{j}\\
	    r &= (\delta_{j},0) ,\quad A = \begin{bmatrix} 0 & 0\\ 0 & 0 \end{bmatrix}, \text{ for } \beta = -\delta_{j}\\
	    r &= (0,\delta_{j}) ,\quad A = \begin{bmatrix} 0 & 0\\ 0 & 0 \end{bmatrix}, \text{ for } \beta = \delta_{j};
	\end{align*}
and,
\begin{align*}
	    H_{2\delta_{j}} \leftrightarrow 
	    \begin{bmatrix} 0 & r \\ c & A \end{bmatrix},
	    \end{align*}
where
	    \begin{align*}
	  r = (0,0),\quad A = \begin{bmatrix} E_{jj} & 0\\ 0 & -E_{jj} \end{bmatrix}.
	    \end{align*}
Note that we will furthermore write $T_{\pm i}$ for $T_{\pm \delta_{i}}$.

    \subsubsection{The adjoint representation}

    The even part $\mathfrak{osp}(1|2n)_{\bar{0}}$ of $\mathfrak{osp}(1|2n)$ is a Lie algebra isomorphic to $\mathfrak{sp}(2n)$, and the odd part  $\mathfrak{osp}(1|2n)_{\bar{1}}$ is isomorphic to the natural $\mathfrak{sp}(2n)$-module $\mathbb{C}^{2n}$. Choosing a Cartan subalgebra $\mathfrak{h}$ of $\mathfrak{osp}(1|2n)$ is precisely a matter of choosing a Cartan subalgebra $\mathfrak{h}$ of $\mathfrak{osp}(1|2)_{\bar{0}}$, which we take to be $\sum\limits_{i=1}^{n} \mathbb{C}H_{2\delta_{j}}$. Then $\Ad$ is the usual adjoint representation for the action of the Cartan subalgebra on $\mathfrak{osp}(1|2n)$. Roots $\beta$ are defined accordingly. Denote a root vector of $\beta$ by $X_{\beta}$. Then $Ad(H_{2\delta_{j}})(X_{\beta}) = \beta_{j}X$ for all $j$. The set given in \eqref{osp-basis-sym} is a basis of root vectors (along with Cartan elements) and identifies a complete set of roots of $\mathfrak{osp}(1|2n)$. Note that each root space $\mathbb{C}X_{\beta}$ is $1$-dimensional and $[X_{\alpha}, X_{\beta}] \in \mathbb{C}X_{\alpha + \beta}$, where $\mathbb{C}X_{\alpha + \beta} = \{0\}$ whenever $\alpha + \beta$ is not a root. 
     
    \subsubsection{Parity of roots of $\mathfrak{osp}(1|2n)$}\label{subsubsec:Roots}
    The nonreduced $BC_n$ root system $\Phi$ of $\mathfrak{osp}(1|2n)$ is given by the union of even (bosonic) roots $\Phi_{0} = \{\pm2\delta_{j}, \pm\delta_{i}\pm\delta_{j}, \delta_{i} - \delta_{j} \mid 1 \leq i \neq j \leq n \}$ with odd (fermionic) roots $\Phi_{1} = \{\pm\delta_{j} \mid 1 \leq j \leq n\}$. Moreover, we have the following set of positive roots $\Phi^{+} = \{ -\delta_{\ell}, -2\delta_{\ell}, \delta_{i} - \delta_{j}, -\delta_{i} - \delta_{j} \mid 1 \leq \ell \leq n, 1 \leq i < j \leq n\}$ associated to the non-standard choice of base $\Pi = \{ -\delta_{1}, \delta_{j} - \delta_{j+1} \mid 1 \leq  j < n \}$ for $\Phi$. The association: Positive roots are roots which are positive integral sums (with more than one summand) of elements of $\Pi$ or are they themselves elements of $\Pi$, the simple roots. The negative roots $\Phi^{-}$ are equal to $-\Phi^{+}$. Positive even (respectively, odd) roots, as well as their negative counterparts, are defined with the appropriate intersection with $\Phi_{0}$ (respectively, $\Phi_{1}$). A root vector is exclusively positive even, positive odd, negative even, or negative odd.

    In particular, $\mathfrak{osp}(1|2n)$ has a triangular decomposition: 
    \[\mathfrak{osp}(1|2n) = \mathfrak{n}_{-} \oplus \mathfrak{h} \oplus \mathfrak{n}_{+}, \text{ with } \mathfrak{n}_{\mp} = \sum_{1 \leq i \neq j \leq n}\mathbb{C}X_{\pm 2\delta_{j}} \oplus \mathbb{C}X_{\pm \delta_{j}} \oplus \mathbb{C}X_{\pm\delta_{i} \pm \delta_{j}} \oplus \mathbb{C}X_{\mp\delta_{i} \pm \delta_{j}}.\]  

    \subsubsection{Oscillator representation}
 
 In \cite{mussonLieSuperalgebrasAssociated1999}, we see that the universal enveloping algebra $U(\mathfrak{osp}(1|2n))$ of $\mathfrak{osp}(1|2n)$ surjects onto the $n$th Weyl superalgebra $A_{n}$, the algebra of polynomial differential operators. In particular, $A_{n}$ is generated as a superalgebra by the operators 
	\begin{align} 
	x_{j}\colon& \mathbb{C}[\mathbf{x}] \rightarrow \mathbb{C}[\mathbf{x}],\thinspace f \mapsto x_{j}f \label{op:commx}\\ 
	\partial_{x_{j}}\colon& \mathbb{C}[\mathbf{x}] \rightarrow \mathbb{C}[\mathbf{x}],\thinspace f \mapsto \frac{\partial f}{\partial x_{j}} \label{op:commdelx}.
	\end{align}
	The defining relation is the commutator $\partial_{x_{j}}x_{j} - x_{j}\partial_{x_{j}} = 1_{A_{n}}$, and we get a consistent $\mathbb{Z}_{2}$-grading from the following integral grading on $A_{n}$. Let $\deg{(x_{j})} = 1$,
	 $\deg{(\partial_{x_{j}})} = -1$, and extend multiplicatively such that the degree of the monomial $x_{i_{1}}x_{i_{2}}\cdots x_{i_{k}} \partial_{j_{1}}\partial_{j_{2}}\cdots \partial_{j_{l}}$ is \[\deg{(x_{i_{1}}x_{i_{2}}\cdots x_{i_{k}} \partial_{j_{1}}\partial_{j_{2}}\cdots \partial_{j_{l}}}) = k - l.\] 
	Now a superalgebra homomorphism $\phi\!:\U(\mathfrak{osp}(1|2n)) \rightarrow A_{n}$ is induced from the universal nature of $\U(\mathfrak{osp}(1|2))$ and the assignment on $\mathfrak{osp}(1|2)$ below:
	\begin{align}
	    X_{-\delta_{i} - \delta_{j}}  &\mapsto -\partial_{x_{i}}\partial_{x_{j}}, \quad i \neq j\\
	    X_{-2\delta_{j}} &\mapsto -\frac{1}{2}\partial_{x_{j}}^{2}\\
	    X_{\delta_{i} + \delta_{j}} &\mapsto x_{i}x_{j}, \quad i \neq j\\
	    X_{2\delta_{j}} &\mapsto \frac{1}{2}x_{j}^{2}\\
	    X_{\delta_{i} - \delta_{j}} &\mapsto x_{i}\partial_{x_{j}}, \quad i \neq j\\
	    X_{-\delta_{j}} &\mapsto \frac{1}{\sqrt{2}}\partial_{x_{j}}\\
	    X_{\delta_{j}} &\mapsto \frac{1}{\sqrt{2}}x_{j}\\
	    H_{2\delta_{j}} &\mapsto x_{i}\partial_{x_{i}} + \frac{1}{2}.
	\end{align}
Consequently, the super vector space of polynomials $\mathbb{C}[\mathbf{x}]$, with the parity of $\mathbf{x}^{\mathbf{k}}$ given by $\deg(\mathbf{x}^{\mathbf{k}}) \mod{2}$, is an oscillator representation of $\mathfrak{osp}(1|2n)$ via the map $\phi$.
	
Note that we will abbreviate $\partial_{x_{j}}$ by $\partial_{j}$. 
	 
\section{Decomposing \texorpdfstring{$\mathbb{C}[\mathbf{x}] \otimes \mathbb{C}^{1|2n}$}{polynomials tensor standard representation}}\label{sec:DecomposingTPR}	

Unless otherwise stated, the use of $\mathfrak{g}$ is to mean the Lie superalgebra $\mathfrak{osp}(1|2n)$. From here on we will omit the word super when it is not essential.

\subsection{The space of singular vectors in \texorpdfstring{$\mathbb{C}[\mathbf{x}] \otimes \mathbb{C}^{1|2n}$}{polynomials tensor standard representation}}\label{sec:space-singular-vectors}

We define a \emph{$\mathfrak{n}_{+}$-singular vector} $v$ in a highest-weight $\mathfrak{g}$-module $V$ to be a solution to the system of equations $ev = 0$, with $e \in \mathfrak{n}_{+}$, and let $V_{\mathfrak{n}_{+}}^{+} \subset V$ denote the subspace of $\mathfrak{n}_{+}$-singular vectors. The descriptor/prefix/subscript $\mathfrak{n}_{+}$ will be suppressed in most cases.
 
	With the choice of root base $\Pi$ and Cartan subalgebra $\mathfrak{h}$ (hence the choice of a Borel subalgebra) made in Section \ref{sec:Background}, the $\mathfrak{g}$-module $\mathbb{C}[\mathbf{x}]$ is a highest-weight module with the polynomial $1$ as a highest-weight vector of weight $\wt(1) = \frac{1}{2}\nu_{n}$. Similarly, the $\mathfrak{g}$-module $\mathbb{C}^{1|2n}$ is a highest-weight module with $v_{2n}$ as a highest-weight vector of weight $\wt(v_{2n}) =  -\delta_{n}$. 

	\setlength{\tabcolsep}{2pt}
\begin{figure}[H]	
\centering
	\begin{tabular}{c|c|c}
	 $\mathbb{C}[\mathbf{x}]$  & $\mathbb{C}^{1|2n}$ & $\mathbb{C}[\mathbf{x}] \otimes \mathbb{C}^{1|2n}$  \\ \hline
	 $\wt(\mathbf{x}^{\mathbf{k}}) = \mathbf{k} + \frac{1}{2}\nu_{n} \thinspace$ &
	 $\thinspace \wt(v_{j}) = \begin{cases}  \phantom{\hspace{0.13cm} -} 0,& \hfill j = 0\\ \phantom{-}    \delta_{j},&  \hfill 1\leq j \leq n\\ -\delta_{j},&  \hfill n+1 \leq j \leq 2n  \end{cases}$
& $\thinspace \wt(f \otimes v_{j}) = \wt(f) + \wt(v_{j})$
	\end{tabular}
         \caption{Weights of weight vectors in relevant $\mathfrak{g}$-modules}%
\label{fig:table1}
\end{figure}
    
	Consequently, $1 \otimes v_{2n}$ of weight $\frac{1}{2}\nu_{n} - \delta_{n}$ is a singular vector in 
	\[ \mathbb{C}[\mathbf{x}] \otimes \mathbb{C}^{1|2n} = \bigoplus^{\text{finite}}_{i} U(\mathfrak{g})w_{i} = \bigoplus^{\text{finite}}_{i} U(\mathfrak{n}_{-})w_{i}, \]
where, letting $\{w_{i}\}$ be a basis of $V^{+}$, each $U(\mathfrak{n}_{-})w_{i}$ is a unique irreducible $\mathfrak{g}$-submodule of $\mathbb{C}[\mathbf{x}] \otimes \mathbb{C}^{1|2n}$ generated by the singular vector $w_{i}$.

\begin{Remark}[Dimension of space of singular vectors]\label{remark:singular-vectors}
We now claim the following vectors in $\mathbb{C}[\mathbf{x}]\otimes\mathbb{C}^{1|2n}$ span the space annihilated by the odd positive roots vectors in $\mathfrak{g}$:
	\begin{align*}
1 \otimes v_{n+j}, \quad 1 \otimes v_{0} + \sqrt{2} \sum\limits_{j=1}^{n} x_{j} \otimes v_{n+j}, \quad -\sqrt{2} x_{j} \otimes v_{0} + 1 \otimes v_{j} + x_{j}^{2} \otimes v_{n+j},
	\end{align*}
	with $1 \leq j \leq n$.
On the other hand, $-\sqrt{2} x_{j} \otimes v_{0} + 1 \otimes v_{j} + x_{j}^{2} \otimes v_{n+j}$ is not annihilated by the positive even root vectors in $\mathfrak{g}$.

Since the result is not new, but a minor correction to a published computation, we leave the related justification to the appendix as a proof of \Cref{prop:space-of-singular-vectors}. The main takeaway is that  
\begin{equation*}
\dim{V_{n}^{+}} = \begin{cases}
2, &\hfill n > 1\\
3, &\hfill n = 1.
\end{cases}
\end{equation*}
\end{Remark}

\subsection{The \texorpdfstring{$n>1$}{n greater than one} case}
We now fix $n \in \mathbb{Z}_{>1}$ and let $V = \mathbb{C}[\mathbf{x}] \otimes \mathbb{C}^{1|2n}$. We define $V^{\text{lil}} = \mathbb{C}[\mathbf{x}] \otimes \mathbb{C}v_{0}$ and  $V^{\text{big}} = \mathbb{C}[\mathbf{x}] \otimes (\underset{i > 0}{\oplus} \mathbb{C}v_{i})$. Note that $V = V^{\text{lil}} \oplus V^{\text{big}}$ is a subspace decomposition of $V$ since $V_{\bar{i}} \supset V^{\text{lil}}_{\bar{i}} = \mathbb{C}[\mathbf{x}]_{\bar{i}} \otimes \mathbb{C}v_{0}$ and $V_{\bar{i}} \supset V^{\text{big}}_{\bar{i}} = \mathbb{C}[\mathbf{x}]_{\bar{i}+\bar{1}} \otimes (\underset{i > 0}{\oplus} \mathbb{C}v_{})$ for $\bar{i} \in \mathbb{Z}_{2}$.
Unfortunately, neither $V^{\text{lil}}$ nor $V^{\text{big}}$ are $\mathfrak{g}$-submodules of $\mathbb{C}[\mathbf{x}] \otimes \mathbb{C}^{1|2n}$; for instance, $X_{-\delta_{1}}(1 \otimes v_{0}) = -1 \otimes v_{n+1} \notin V^{\text{lil}}$ and $X_{-\delta_{1}}(1 \otimes v_{1}) = 1 \otimes v_{0} \notin V^{\text{big}}$. However we do have a weight basis for $V$ which we use in constructing automorphisms of $V$ that will serve as intertwining operators. We make the following conventions to express elementary tensors in $V$: 
$Y_{\mathbf{k}} = \mathbf{x}^{\mathbf{k}} \otimes v_{0}$ and $Z_{\mathbf{k},i} = \mathbf{x}^{\mathbf{k}} \otimes v_{i}$.

\subsubsection{Weight basis of \texorpdfstring{$V$}{V}} 
Let $\Lambda$ be the set of weights of $V$. We recall that weights of tensor products of weight modules are sums of weights of the tensor product factors. Thus $\lambda \in \Lambda$ is a sum of the form $\mathbf{k} +1/2\nu_{n} + d$ for $\mathbf{k} \in \mathbb{Z}^{n}_{\geq 0}$ and $d \in \{-\delta_{j},\mathbf{0},\delta_{j} \mid 1 \leq j \leq n\}$. Noting that $ \mathbf{k} + \mathbf{0}$ and $\mathbf{k} + \delta_{j}$ are elements of $\mathbb{Z}_{\geq 0}^{n}$ as much as $\mathbf{k}$ is, we have that $\lambda \in \Lambda$ is a sum of the form $\mathbf{k} +1/2\nu_{n} - \delta_{j}$ for some $\mathbf{k} \in \mathbb{Z}_{\geq 0}$ and $1 \leq j \leq n$.  Additionally, it will be helpful to part $\Lambda$ by those $\lambda = (\lambda_{1},\lambda_{2}, \ldots, \lambda_{n})$ satisfying $\lambda_{i} = -\frac{1}{2}$, for some $i$, and the remaining weights. In the latter case, each weight is associated with a nonnegative integer $C = C(\lambda)$ which counts the components $\lambda_{i_{1}}, \lambda_{i_{2}}, \ldots, \lambda_{i_{C}}$, if any, not equal to $\frac{1}{2}$. Denote by $\Lambda_{C}$ the set of weights with precisely $C$ non-one-half components (assuming each component is greater than 0) and by $\Lambda^{(j)}$ the set of weights of the form $\mathbf{0} +\frac{1}{2}\nu_{n} - \delta_{j}$, for which $\lambda_{j} = -\frac{1}{2}$. Then $\{\Lambda^{(j)}, \Lambda_{C} \mid 1 \leq j \leq n, \thickspace 0 \leq C \leq 2n+1 \}$ is a partition of $\Lambda$.  

Since $\lambda \in \Lambda_{C}$ implies $\mathbf{k} - \delta_{j} \in \mathbb{Z}^{n}_{\geq 0}$, we can write $\lambda = \mathbf{k} + \frac{1}{2}\nu_{n}$. In other words, $\mathbf{k} = \mathbf{0}$ for $\lambda \in \Lambda^{(j)}$; otherwise, $\mathbf{k} = \lambda - \frac{1}{2}\nu_{n}$.
Furthermore, the $\lambda$-weight space $V_{\lambda}$ can be assigned a basis $\mathcal{B}_{\lambda}$:
\[\mathcal{B}_{\lambda} = \mathcal{B}_{(\lambda_{1},\lambda_{2},\ldots,\lambda_{n})} = 
\begin{cases} 
    \{Z_{\mathbf{0},n+j}\}, &  \lambda \in \Lambda^{(j)}\\ 
    \{ Y_{\mathbf{0}}, Z_{\delta_{j},n+j} \mid 1 \leq j \leq n \}, & \lambda \in \Lambda_{0}\\
    \{ Y_{\mathbf{k}}, Z_{\mathbf{k}+\delta_{j},n+j}, Z_{\mathbf{k}-\delta_{i_{l}},i_{l}}\mid 1 \leq j \leq n, \thinspace 1 \leq l \leq C \}, & \lambda \in \Lambda_{C \neq 0} \text{ with } \lambda_{i_{l}} \neq \frac{1}{2}.\\
\end{cases}
\]
By dimension:
\[\dim(\mathcal{B}_{\lambda}) = 
\begin{cases}
    1, &  \lambda \in \Lambda^{(j)}\\ 
    C+n+1, & \lambda \in \Lambda_{C}.\\
\end{cases}
\]
Then bases for $V^{\text{lil}}$ and $V^{\text{big}}$ are given by the $Y_{\mathbf{k}}$ and $Z_{\mathbf{k},i}$ (including $Z_{\mathbf{0},n+j}$), respectively:
\[V = V^{\text{lil}} \oplus V^{\text{big}} = \big(\bigoplus_{\mathbf{k}\in \mathbb{Z}_{\geq 0}^{n}}\mathbb{C}Y_{\mathbf{k}}\big) \oplus \big(\bigoplus_{j=1}^{n} \mathbb{C}Z_{\mathbf{0},n+j} \bigoplus_{\mathbf{k}\in \mathbb{Z}_{>0}^{n}}\mathbb{C}Z_{\mathbf{k},i}\big).\]
The $Y$-basis and the $Z$-basis are indeed the appropriate subsets of the standard basis of $V$ comprising elementary tensors.

We recall \Cref{remark:singular-vectors} and the preceding sentence: The singular vector $w_{1} = 1 \otimes v_{2n}$ is of weight $\mathbf{0} + \frac{1}{2}\nu_{n} - \delta_{n} = \lambda_{w_{1}} \in \Lambda^{(n)}$, and the singular vector $w_{2} = 1 \otimes v_{0} + \sqrt{2}\sum\limits_{i=1}^{n}x_{i} \otimes v_{n+i}$ is of weight $\mathbf{0} + \frac{1}{2}\nu_{n} = \lambda_{w_{2}} \in \Lambda_{0}$. So $U(\mathfrak{n}_{-})w_{2}$, being an irreducible highest-weight module, is isomorphic to $\mathbb{C}[x_{1},x_{2},\ldots,x_{n}]$. Now $V^{\text{lil}}$ is isomorphic to $U(\mathfrak{n}_{-})w_{2}$ as super vector spaces. We are motivated to find an isomorphism that commutes with the action of the generators of $\mathfrak{g}$ in $U(\mathfrak{n}_{-})w_{2}$, that is, an intertwining automorphism. In particular, we seek to describe a $\mathfrak{g}$-action on $V^{\text{lil}}$ with respect to the $Y$-basis and a similar $\mathfrak{g}$-action on $V^{\text{big}}$ with respect to the $Z$-basis. 

\subsubsection{Intertwiners}

If $\Gamma$ is an intertwining automorphism as described above, then $\Gamma$ should preserve weight spaces. In particular, $X_{-\delta_{j}}\Gamma(\alpha_{\mathbf{0}}Y_{\mathbf{0}}) = X_{-\delta_{j}}(w_{2}) = 0$ for $1 \leq j \leq n$, $\alpha_{\mathbf{0}} \in \mathbb{C}$. More generally, $X_{-\delta_{j}}\Gamma(\alpha_{\mathbf{k}}Y_{\mathbf{k}}) \in \mathbb{C}X_{-\delta_{j}}X_{\delta_{\mathbf{k}}}(w_{2})$, where $X_{\delta_{\mathbf{k}}} \in U(\mathfrak{g})$ is equal to $X_{\delta_{k_{1}}}X_{\delta_{k_{2}}} \cdots X_{\delta_{k_{n}}}$ (or any permutation of the factors), letting $X_{\delta_{0}} = 1$. The previous line is justified by an analysis of the weights: $\wt\left(X_{\delta_{\mathbf{k}}}(w_{2})\right) = \delta_{k_{1}} + \delta_{k_{2}} + \cdots + \delta_{k_{n}} + \wt(w_{2}) = \wt\left(\sigma\cdot X_{\delta_{k_{1}}}X_{\delta_{k_{2}}} \cdots X_{\delta_{k_{n}}}(w_{2})\right)$ for $\sigma \in \mathfrak{S}_{n}$, the permutation group acting on the indices of $\mathbf{k}$. Then consider 
\begin{equation*}
  \Gamma = \mathbbm{1} \otimes \mathbbm{1} + \text{sum of elementary tensors}, \label{Ansatz}\tag{Ansatz} 
\end{equation*} 
where the elementary tensors are formed by pairs of lowering and raising operators based on comparing the right-hand side of the following lines:

\begin{align*}   
    X_{-\delta_{r}}X_{\delta_{q}}X_{\delta_{p}}(w_{2}) & = X_{-\delta_{r}}X_{\delta_{q}}\left(\frac{1}{\sqrt{2}} m_{p} \otimes \mathbbm{1} + \mathbbm{1} \otimes T_{p} \right ) (1 \otimes v_{0} + \sqrt{2}\sum\limits_{\ell=1}^{n}x_{\ell} \otimes v_{n+\ell}) \\ 
    & = X_{-\delta_{r}}\left(\frac{1}{\sqrt{2}} m_{q} \otimes \mathbbm{1} + \mathbbm{1} \otimes T_{q} \right )(-\frac{1}{\sqrt{2}} x_{p} \otimes v_{0} + \sum\limits_{\ell=1}^{n}x_{p}x_{\ell} \otimes v_{n+\ell} + 1 \otimes v_{p})\\ 
    & = \left(\frac{1}{\sqrt{2}} \partial_{r} \otimes \mathbbm{1} + \mathbbm{1} \otimes T_{-r}\right)(\frac{1}{2}x_{q}x_{p} \otimes v_{0} + \frac{1}{\sqrt{2}}\sum\limits_{\ell=1}^{n}x_{q}x_{p}x_{\ell} \otimes v_{n+\ell}\\
    & \quad + \frac{1}{\sqrt{2}}x_{p} \otimes v_{q} + \frac{1}{\sqrt{2}}x_{q} \otimes v_{p})\\
     & = \frac{1}{2\sqrt{2}}(\delta_{rq}x_{p} + \delta_{rp}x_{q}) \otimes v_{0} - \frac{1}{2} x_{q}x_{p} \otimes v_{n+r}\\ 
     & \quad + \frac{1}{2}\sum\limits_{\ell=1}^{n}(\delta_{rq}x_{p}x_{\ell} + \delta_{rp}x_{q}x_{\ell}+\delta_{r\ell}x_{q}x_{p}) \otimes v_{n+\ell} + \frac{1}{2}\delta_{rq} \otimes v_{p} + \frac{1}{2}\delta_{rp} \otimes v_{q} \\ 
    &= \frac{1}{2\sqrt{2}}(\delta_{rq}x_{p} + \delta_{rp}x_{q}) \otimes v_{0} + \frac{1}{2}\sum\limits_{\ell=1}^{n}(\delta_{rq}x_{p}x_{\ell} + \delta_{rp}x_{q}x_{\ell}) \otimes v_{n+\ell}\\
    & \quad + \frac{1}{2}\delta_{rq} \otimes v_{p} + \frac{1}{2}\delta_{rp} \otimes v_{q} \\ 
    \intertext{with}
   X_{-\delta_{r}}(\alpha_{\delta_{q} + \delta_{p}}x_{q}x_{p} \otimes v_{0}) &= \left(\frac{1}{\sqrt{2}} \partial_{r} \otimes \mathbbm{1} + \mathbbm{1} \otimes T_{-r} \right)(\alpha_{\delta_{q} + \delta_{p}}x_{q}x_{p} \otimes v_{0})\\
   & = \frac{\alpha_{\delta_{q}+\delta_{p}}}{\sqrt{2}}(\delta_{rq}x_{p} + \delta_{rp}x_{q}) \otimes v_{0} - \alpha_{\delta{q}+\delta_{p}}x_{q}x_{p} \otimes v_{n+r}.
\end{align*}

A similar investigation of weight spaces with regards to the the $Z$-basis leads to the following morphisms of $V$:
\begin{align}
\prescript{w_{1}}{}{\Gamma} &= \mathbbm{1} \otimes \mathbbm{1} - \sqrt{2}\sum\limits_{i=1}^{n} \partial_{x_{i}} \otimes T_{i} + \sqrt{2}\sum\limits_{i=1}^{n} x_{i} \otimes T_{-i} \label{auto-Z}\\
\prescript{w_{2}}{}{\Gamma} &= \mathbbm{1} \otimes \mathbbm{1} + \sqrt{2}\sum\limits_{i=1}^{n} \partial_{x_{i}} \otimes T_{i} - \sqrt{2}\sum\limits_{i=1}^{n} x_{i} \otimes T_{-i} \label{auto-Y}.
\end{align}

\subsubsection{Automorphisms of \texorpdfstring{$V$}{V}}

We show that \Cref{auto-Y,auto-Z} give automorphisms of $V$ by considering their restrictions to each $\lambda$-weight space $V_{\lambda}$. Label the corresponding (square) matrices of the restrictions with respect to $\mathcal{B}_{\lambda}$ by $\prescript{w_{1}}{}{\Gamma}_{\lambda}$ and $\prescript{w_{2}}{}{\Gamma}_{\lambda}$. Note the \eqref{Ansatz} above preserves weight spaces by the pairing of lowering and raising operators in elementary tensors. In particular, $\prescript{w_{1}}{}{\Gamma}$ and $\prescript{w_{2}}{}{\Gamma}$ compute with the image of the Cartan subalgebra under the tensor product representation $V$. It suffices to show $\prescript{w_{1}}{}{\Gamma}_{\lambda}$ and $\prescript{w_{2}}{}{\Gamma}_{\lambda}$ are nonsingular in order to establish \eqref{auto-Z} and \eqref{auto-Y} give automorphisms. 

\begin{Remark}[Notation for the matrices] 
In the matrices below, we let dotted segments denote constant entries between equally valued entries and dotted polygons signify a region of constant entries of the same value as vertices. The entry $a$ equals $(-1)^{|\mathbf{k}|}\sqrt{2}$, recalling $\mathbf{k} = \lambda - \frac{1}{2}\nu_{n}$ (for relevant $\lambda$).
\end{Remark}

\begin{align*}
\prescript{w_{1}}{}{\Gamma}_{\lambda} &= 
\begin{dcases}
\begin{tikzpicture}[baseline=-\the\dimexpr\fontdimen22\textfont2
\relax ]
 \tikzset{BarreStyle/.style =   {opacity=.6,line width=0.5 mm,line cap=round,color=#1}}
\matrix[matrix of math nodes,left delimiter = (,right delimiter = ),row sep=10pt,column sep = 10pt, ampersand replacement=\&] (m) {
1\\
};
\end{tikzpicture}, & \lambda \in \Lambda^{(j)}\\
\begin{tikzpicture}[baseline=-\the\dimexpr\fontdimen22\textfont2
\relax ]
 \tikzset{BarreStyle/.style =   {opacity=.6,line width=0.5 mm,line cap=round,color=#1}}
\matrix[matrix of math nodes,left delimiter = (,right delimiter = ),row sep=10pt,column sep = 10pt, ampersand replacement=\&] (m) {
1 \& a \& \& \& a\\
a \& \&0  \& \&0\\
\& 0\& \& \&  \\
\& \& \& \& 0\\
a \&0 \& \& 0\& 1\\
};
\draw[loosely dotted, thick] (m-1-1)-- (m-5-5); 
\draw[loosely dotted, thick] (m-1-2)-- (m-1-5); 
\draw[loosely dotted, thick] (m-2-1)-- (m-5-1); 
\draw[loosely dotted, thick] (m-2-3)-- (m-2-5); 
\draw[loosely dotted, thick] (m-2-3)-- (m-4-5); 
\draw[loosely dotted, thick] (m-2-5)-- (m-4-5); 
\draw[loosely dotted, thick] (m-3-2)-- (m-5-4); 
\draw[loosely dotted, thick] (m-3-2)-- (m-5-2); 
\draw[loosely dotted, thick] (m-5-2)-- (m-5-4); 
\end{tikzpicture}, & \lambda \in \Lambda_{0}\\
\begin{tikzpicture}[baseline=-\the\dimexpr\fontdimen22\textfont2
\relax, decoration=brace]
 \tikzset{BarreStyle/.style =   {opacity=.6,line width=0.5 mm,line cap=round,color=#1}}
\matrix[matrix of math nodes,left delimiter = (,right delimiter = ),row sep=10pt,column sep = 10pt, ampersand replacement=\&] (m) {
1 \& -a \& \& -a \& a \& \& a\\
-a \&   \& 0\&   \&    \&  \&0 \\
\& 0\& \& \& \& \& \\
-a \&  \& \&   \&  \&  \&\\
a \&   \& \&   \&    \&  \& \\
\&\&\&\&\&\&0\\
a \& 0 \& \&   \&  \&  0\&1\\
};
\draw[loosely dotted, thick] (m-1-1)-- (m-7-7); 
\draw[loosely dotted, thick] (m-1-2)-- (m-1-4); 
\draw[loosely dotted, thick] (m-1-5)-- (m-1-7); 
\draw[loosely dotted, thick] (m-2-1)-- (m-4-1); 
\draw[loosely dotted, thick] (m-5-1)-- (m-7-1); 
\draw[loosely dotted, thick] (m-2-3)-- (m-2-7); 
\draw[loosely dotted, thick] (m-2-3)-- (m-6-7); 
\draw[loosely dotted, thick] (m-2-7)-- (m-6-7); 
\draw[loosely dotted, thick] (m-3-2)-- (m-7-6); 
\draw[loosely dotted, thick] (m-3-2)-- (m-7-2); 
\draw[loosely dotted, thick] (m-7-2)-- (m-7-6); 
\draw[decorate,transform canvas={yshift=-.1em},thick] (m-1-5.north east) -- node[above=2pt] {$C$-many} (m-1-7.north west); 
\end{tikzpicture}, & \lambda \in \Lambda_{C \neq 0} 
\end{dcases}
\end{align*}
\begin{align*}
\prescript{w_{2}}{}{\Gamma}_{\lambda} &= 
\begin{dcases}
\begin{tikzpicture}[baseline=-\the\dimexpr\fontdimen22\textfont2
\relax ]
 \tikzset{BarreStyle/.style =   {opacity=.6,line width=0.5 mm,line cap=round,color=#1}}
\matrix[matrix of math nodes,left delimiter = (,right delimiter = ),row sep=10pt,column sep = 10pt, ampersand replacement=\&] (m) {
1\\
};
\end{tikzpicture}, & \lambda \in \Lambda^{(j)}\\
\begin{tikzpicture}[baseline=-\the\dimexpr\fontdimen22\textfont2
\relax ]
 \tikzset{BarreStyle/.style =   {opacity=.6,line width=0.5 mm,line cap=round,color=#1}}
\matrix[matrix of math nodes,left delimiter = (,right delimiter = ),row sep=10pt,column sep = 10pt, ampersand replacement=\&] (m) {
1 \& a \& \& \& a\\
a \& \&0  \& \&0\\
\& 0\& \& \&  \\
\& \& \& \& 0\\
a \&0 \& \& 0\& 1\\
};
\draw[loosely dotted, thick] (m-1-1)-- (m-5-5); 
\draw[loosely dotted, thick] (m-1-2)-- (m-1-5); 
\draw[loosely dotted, thick] (m-2-1)-- (m-5-1); 
\draw[loosely dotted, thick] (m-2-3)-- (m-2-5); 
\draw[loosely dotted, thick] (m-2-3)-- (m-4-5); 
\draw[loosely dotted, thick] (m-2-5)-- (m-4-5); 
\draw[loosely dotted, thick] (m-3-2)-- (m-5-4); 
\draw[loosely dotted, thick] (m-3-2)-- (m-5-2); 
\draw[loosely dotted, thick] (m-5-2)-- (m-5-4); 
\end{tikzpicture}, & \lambda \in \Lambda_{0}\\
\begin{tikzpicture}[baseline=-\the\dimexpr\fontdimen22\textfont2
\relax, decoration=brace]
 \tikzset{BarreStyle/.style =   {opacity=.6,line width=0.5 mm,line cap=round,color=#1}}
\matrix[matrix of math nodes,left delimiter = (,right delimiter = ),row sep=10pt,column sep = 10pt, ampersand replacement=\&] (m) {
1 \& a \& \& a \& -a \& \& -a\\
a \&   \& 0\&   \&    \&  \&0 \\
\& 0\& \& \& \& \& \\
a \&  \& \&   \&  \&  \&\\
-a \&   \& \&   \&    \&  \& \\
\&\&\&\&\&\&0\\
-a \& 0 \& \&   \&  \&  0\&1\\
};
\draw[loosely dotted, thick] (m-1-1)-- (m-7-7); 
\draw[loosely dotted, thick] (m-1-2)-- (m-1-4); 
\draw[loosely dotted, thick] (m-1-5)-- (m-1-7); 
\draw[loosely dotted, thick] (m-2-1)-- (m-4-1); 
\draw[loosely dotted, thick] (m-5-1)-- (m-7-1); 
\draw[loosely dotted, thick] (m-2-3)-- (m-2-7); 
\draw[loosely dotted, thick] (m-2-3)-- (m-6-7); 
\draw[loosely dotted, thick] (m-2-7)-- (m-6-7); 
\draw[loosely dotted, thick] (m-3-2)-- (m-7-6); 
\draw[loosely dotted, thick] (m-3-2)-- (m-7-2); 
\draw[loosely dotted, thick] (m-7-2)-- (m-7-6); 
\draw[decorate,transform canvas={yshift=-.1em},thick] (m-1-5.north east) -- node[above=2pt] {$C$-many} (m-1-7.north west);
\end{tikzpicture}, & \lambda \in \Lambda_{C \neq 0}. 
\end{dcases}
\end{align*}

The matrices $\prescript{w_{1}}{}{\Gamma}_{\lambda}, \prescript{w_{2}}{}{\Gamma}_{\lambda}$ equal the $1 \times 1$ identity for $\lambda \in \Lambda^{(j)}$ or, whenever $\lambda \in \Lambda_{C}$, come as particular arrowhead matrices upon which elementary row operations yield lower triangular matrices with diagonal $\diag(1-(n+C)a^{2},1,\ldots,1)$. So the determinant of $\prescript{w_{2}}{}{\Gamma}_{\mathbf{k}}$ equals the determinant of $\prescript{w_{1}}{}{\Gamma}_{\lambda}$, which is either $1$ or $1-(n+C)a^{2} \neq 0$, as $a$ takes values $\pm\sqrt{2}$. Thus $\prescript{w_{2}}{}{\Gamma}_{\lambda}$ and $\prescript{w_{1}}{}{\Gamma}_{\lambda}$ are automorphisms of the $\lambda$-weight space of $V$. 
As desired, $\prescript{w_{1}}{}{\Gamma} = \sum_{\lambda \in \Lambda} \prescript{w_{1}}{}{\Gamma}_{\lambda}$ and $\prescript{w_{2}}{}{\Gamma} = \sum_{\lambda \in \Lambda} \prescript{w_{2}}{}{\Gamma}_{\lambda}$  are automorphisms of $V$. The operator $\prescript{w_{1}}{}{\Gamma}$ can be thought of as infinite diagonal block matrix with each block one of the $\prescript{w_{1}}{}{\Gamma}_{\lambda}$ and likewise for $\prescript{w_{2}}{}{\Gamma}$ with respect to the $\prescript{w_{2}}{}{\Gamma}_{\lambda}$.

\subsection{Formulas for action of \texorpdfstring{$\mathfrak{osp}(1|2n)$}{osp(1|2n)} on bases of irreducible summands}\label{subsec:biggerthan1}

The prior discussion leaves us with automorphisms of $V$ which restrict to operators on the weight spaces expressed as arrowhead matrices. That is, we have \Cref{thm:basi-and-actions:arrowhead-gam,thm:basi-and-actions:arrowhead-gamt} of \Cref{thm:basis-and-actions}; thus, we are left to prove \Cref{thm:basis-and-actions:gam-image,thm:basis-and-actions:gamt-image} to complete the proof of the theorem. 
 
\begin{proof}[Proof of \Cref{thm:basis-and-actions:gam-image,thm:basis-and-actions:gamt-image} of \Cref{thm:basis-and-actions}]

Recognizing $V = \mathbb{C}[\mathbf{x}] \otimes \mathbb{C}^{1|2n}$ as a tensor product of $\mathfrak{g}$-modules, define the following operators in $\mathfrak{gl}(V)$:

\begin{align*}
\prescript{w_{2}}{}{\Gamma} &= \mathbbm{1} \otimes \mathbbm{1} + \sqrt{2}\sum\limits_{i=1}^{n} \partial_{x_{i}} \otimes T_{i} - \sqrt{2}\sum\limits_{i=1}^{n} x_{i} \otimes T_{-i},\\
\prescript{w_{1}}{}{\Gamma} &= \mathbbm{1} \otimes \mathbbm{1} - \sqrt{2}\sum\limits_{i=1}^{n} \partial_{x_{i}} \otimes T_{i} + \sqrt{2}\sum\limits_{i=1}^{n} x_{i} \otimes T_{-i}.
\end{align*}

We check the following:

For $1 \leq j \leq n$,
    \begin{align} 
	X_{\delta_{j}}\prescript{w_{2}}{}{\Gamma}(Y_{\mathbf{k}}) 
    & = -\frac{1}{\sqrt{2}}Y_{\mathbf{k}+\delta_{j}} + (-1)^{|\mathbf{k}|}Z_{\mathbf{k},j}  
	 + (-1)^{|\mathbf{k}|}\sum\limits_{i=1}^{n} k_{i}Z_{\mathbf{k}-\delta_{i}+\delta_{j},i} \label{computations:1}\\   
	& \enskip + (-1)^{|\mathbf{k}|}\sum\limits_{i=1}^{n}Z_{\mathbf{k}+\delta_{i}+\delta_{j},n+i}, \notag 
\intertext{whereas}
	\prescript{w_{2}}{}{\Gamma}(Y_{\mathbf{k}+\delta_{j}}) 
	& = Y_{\mathbf{k}+\delta_{j}} - (-1)^{|\mathbf{k}|}\sqrt{2}\sum\limits_{i=1}^{n} (k_{i}+\delta_{ij}) Z_{\mathbf{k}-\delta_{i}+\delta_{j},i} \label{computations:2} \\ 
	& \enskip - (-1)^{|\mathbf{k}|}\sqrt{2}\sum\limits_{i=1}^{n} Z_{\mathbf{k}+\delta_{i}+\delta_{j},n+i}. \notag
\intertext{Continuing,}
	X_{-\delta_{j}}\prescript{w_{2}}{}{\Gamma}(Y_{\mathbf{k}}) \label{computation:3}
	& = -\frac{1}{\sqrt{2}} k_{j} Y_{\mathbf{k}-\delta_{j}} - (-1)^{|\mathbf{k}|}Z_{\mathbf{k},n+j}\\ 
	& \enskip + (-1)^{|\mathbf{k}|}\sum\limits_{i=1}^{n} k_{i}(k_{j} - \delta_{i j})Z_{\mathbf{k}-\delta_{i}+\delta_{j},i} \notag \\
	& \enskip + (-1)^{|\mathbf{k}|}\sum\limits_{i=1}^{n} (k_{j} + \delta_{i j})Z_{\mathbf{k}+\delta_{i}-\delta_{j},n+i}, \notag  
\intertext{whereas}
	\prescript{w_{2}}{}{\Gamma}(Y_{\mathbf{k}-\delta_{j}}) 
	& = Y_{\mathbf{k}-\delta_{j}} + (-1)^{|\mathbf{k} - 1|}\sqrt{2}\sum\limits_{i=1}^{n} (k_{i} - \delta_{ij}) Z_{\mathbf{k}-\delta_{i}+\delta_{j},i}\\
	& \enskip + (-1)^{|\mathbf{k} - 1|}\sqrt{2}\sum\limits_{i=1}^{n} Z_{\mathbf{k}+\delta_{i}-\delta_{j},n+i}. \notag
\intertext{Now for $1 \leq i,j \leq n$,}
	X_{\delta_{j}}\prescript{w_{1}}{}{\Gamma}(Z_{\mathbf{k},i}) 
        & = \frac{1}{\sqrt{2}}Z_{\mathbf{k}+\delta_{j},i} + (-1)^{|\mathbf{k}|}Y_{\mathbf{k}+\delta_{i}+\delta_{j}} - \sqrt{2} Z_{\mathbf{k}+\delta_{i},j},
\intertext{whereas}
	\prescript{w_{1}}{}{\Gamma}(Z_{\mathbf{k}+\delta_{j},i} - Z_{\mathbf{k}+\delta_{i},j}) 
	& = Z_{\mathbf{k}+\delta_{j},i} - (-1)^{|\mathbf{k}|}\sqrt{2} Y_{\mathbf{k}+\delta_{i}+\delta_{j}} - Z_{\mathbf{k}+\delta_{i},j} \\
 & \enskip + (-1)^{|\mathbf{k}|}\sqrt{2}Y_{\mathbf{k}+\delta_{i}+\delta_{j}} \notag;
\intertext{and,}
	X_{\delta_{j}}\prescript{w_{1}}{}{\Gamma}(Z_{\mathbf{k},n+i}) 
	& = \frac{1}{\sqrt{2}}Z_{\mathbf{k}+\delta_{j},n+i} + \delta_{ij}(-1)^{|\mathbf{k}|}Y_{\mathbf{k}} \\
	& \enskip - (-1)^{|\mathbf{k}|} k_{i}Y_{\mathbf{k}-\delta_{i}+\delta_{j}} + \sqrt{2} k_{i} Z_{\mathbf{k}-\delta_{i},j}, \notag
\intertext{whereas}
	\prescript{w_{1}}{}{\Gamma}(Z_{\mathbf{k}+\delta_{j},n+i} + Z_{\mathbf{k}-\delta_{i},j}) 
	& = Z_{\mathbf{k}+\delta_{j},n+i} + (-1)^{|\mathbf{k}|}\sqrt{2} (k_{i} + \delta_{ij})Y_{\mathbf{k}-\delta_{i}+\delta_{j}}\\  
	& \enskip + \delta_{ij} (-1)^{|\mathbf{k}|}\sqrt{2} Y_{\mathbf{k}-\delta_{i}-\delta{j}}  + Z_{\mathbf{k}-\delta_{i},j} - (-1)^{|\mathbf{k}|}\sqrt{2} Y_{\mathbf{k}-\delta_{i}+\delta_{j}}. \notag
\intertext{Continuing,}
	X_{-\delta_{j}}\prescript{w_{1}}{}{\Gamma}(Z_{\mathbf{k},i}) 
	& = \frac{1}{\sqrt{2}} k_{j} Z_{\mathbf{k}-\delta_{j},i} + \delta_{ij}(-1)^{|\mathbf{k}|}Y_{\mathbf{k}}\\
	& \enskip + (-1)^{|\mathbf{k}|} (k_{j} + \delta_{ij}) Y_{\mathbf{k}+\delta_{i}-\delta_{j}} + \sqrt{2} Z_{\mathbf{k}+\delta_{i},n+j}, \notag
\intertext{whereas}
        \prescript{w_{1}}{}{\Gamma}(Z_{\mathbf{k}-\delta_{j},i} -Z_{\mathbf{k}+\delta_{i},n+j}) 
	& = Z_{\mathbf{k}+\delta_{j},i} - (-1)^{|\mathbf{k}|}\sqrt{2} Y_{\mathbf{k}+\delta_{i}+\delta_{j}}\\
	& \enskip + Z_{\mathbf{k}+\delta_{i},j} - (-1)^{|\mathbf{k}|}\sqrt{2} Y_{\mathbf{k}+\delta_{i}+\delta_{j}}; \notag
\intertext{and,}
	X_{-\delta_{j}}\prescript{w_{1}}{}{\Gamma}(Z_{\mathbf{k},n+i}) 
	& = \frac{1}{\sqrt{2}} k_{j}Z_{\mathbf{k}-\delta_{j},n+i} - (-1)^{|\mathbf{k}|} k_{i} (k_{j} - \delta_{ij})Y_{\mathbf{k}-\delta_{i}-\delta{j}}\\
    & \enskip - \sqrt{2} k_{i} Z_{\mathbf{k}-\delta_{i},n+j}, \notag 
\intertext{whereas}
	\prescript{w_{1}}{}{\Gamma}(Z_{\mathbf{k}-\delta_{j},n+i} + Z_{\mathbf{k}-\delta_{i},n+j}) \label{computations:13}
	& = Z_{\mathbf{k}+\delta_{j},i} +(-1)^{|\mathbf{k}|}\sqrt{2} k_{i} Y_{\mathbf{k}-\delta_{i}+\delta_{j}}\\
        & \enskip + \delta_{ij} (-1)^{|\mathbf{k}|}\sqrt{2} Y_{\mathbf{k}-\delta_{i}+\delta_{j}} + Z_{\mathbf{k}-\delta_{i},j} - (-1)^{|\mathbf{k}|}\sqrt{2} Y_{\mathbf{k}-\delta_{i}+\delta_{j}}. \notag 
	\end{align}

\Crefrange{computations:1}{computations:13} imply the following when noting $Y_{\mathbf{0}-\delta_{j}} = Z_{\mathbf{0} - \delta_{j},i} = Z_{\mathbf{0} - \delta_{j},n+i} = 0$:
\begin{align*}
X_{\delta_{j}}\prescript{w_{2}}{}{\Gamma}(Y_{\mathbf{k}}) &= \prescript{w_{2}}{}{\Gamma}(-\frac{1}{\sqrt{2}}Y_{\mathbf{k} + \delta_{j}})\\
X_{-\delta_{j}}\prescript{w_{2}}{}{\Gamma}(Y_{\mathbf{k}}) &= \prescript{w_{2}}{}{\Gamma}(-\frac{1}{\sqrt{2}} k_{j}Y_{\mathbf{k} - \delta_{j}})\\
X_{\delta_{j}}\prescript{w_{1}}{}{\Gamma}(Z_{\mathbf{k},i}) &= \prescript{w_{1}}{}{\Gamma}(\frac{1}{\sqrt{2}}Z_{\mathbf{k}+\delta_{j},i} - \sqrt{2} Z_{\mathbf{k}+\delta_{i},j})\\
X_{\delta_{j}}\prescript{w_{1}}{}{\Gamma}(Z_{\mathbf{k},n+i}) &= \prescript{w_{1}}{}{\Gamma}(\frac{1}{\sqrt{2}}Z_{\mathbf{k}+\delta_{j},n+i} + \sqrt{2} k_{i} Z_{\mathbf{k}-\delta_{i},j})\\
X_{-\delta_{j}}\prescript{w_{1}}{}{\Gamma}(Z_{\mathbf{k},i}) &= \prescript{w_{1}}{}{\Gamma}(\frac{1}{\sqrt{2}} k_{j} Z_{\mathbf{k}-\delta_{j},i} + \sqrt{2} Z_{\mathbf{k}+\delta_{i},n+j})\\
X_{-\delta_{j}}\prescript{w_{1}}{}{\Gamma}(Z_{\mathbf{k},n+i}) &= \prescript{w_{1}}{}{\Gamma}(\frac{1}{\sqrt{2}} k_{j} Z_{\mathbf{k}-\delta_{j},n+i} - \sqrt{2} k_{i} Z_{\mathbf{k}-\delta_{i},n+j}).
\end{align*}

\Cref{rels-osp} justifies that the first two equations and last four equations in the preceding list fully characterize the $\mathfrak{g}$-module structure of $\prescript{w_{2}}{}{\Gamma}(V^{\text{lil}})$ and $\prescript{w_{1}}{}{\Gamma}(V^{\text{big}})$, respectively. 
\end{proof}

\Cref{cor:decomposition-of-V} follows from \Cref{prop:space-of-singular-vectors}: Since $\prescript{w_{1}}{}{\Gamma}$ and $\prescript{w_{2}}{}{\Gamma}$ are automorphisms and $V = V^{\text{lil}} \oplus V^{\text{big}}$, it is the case that $V$ decomposes as a $\mathfrak{g}$-module with summands $\prescript{w_{2}}{}{\Gamma}(V^{\text{lil}})$ and $\prescript{w_{1}}{}{\Gamma}(V^{\text{big}})$. 

\Cref{cor:gam-osp-action} is an immediate consequence of conjugating the tensor product representation by $\prescript{w_{2}}{}{\Gamma}$ and restricting to $V^{\text{lil}}$; \Cref{cor:gamt-osp-action} follows from conjugation by $\prescript{w_{1}}{}{\Gamma}$ and restriction to $V^{\text{big}}$.

\renewcommand\appendix{\par
 \setcounter{section}{0}%
  \renewcommand{\thesection}{\Alph{section}}
}

\renewcommand{\theHsection}{A\arabic{section}}

\appendix\label{appendix}

  \section{Supermath}\label{appendix:supermath}
    ``Elements of graded algebra" is the title of Chapter 1 in \cite{bartocciGeometrySupermanifolds1991}, in which ``graded $\equiv \mathbb{Z}_{2}$". The referenced book is a recommended source to find definitions and examples of the super objects presently under study. It is now common that $\mathbb{Z}_{2}$-graded objects are identified by appending super to the name of a known object; however, this is dangerous as the category of $\mathbb{Z}_{2}$-graded vector spaces and the category of super vector spaces share objects but have distinct structures as symmetric monoidal categories resulting from a difference in braiding. Namely, for super vector spaces $V, W$, we have $V \otimes W \cong W \otimes V$ by $v \otimes w \mapsto \epsilon(v,w) w \otimes v$, where $\epsilon$ is given by "Manin's rule of signs" \cite{maninIntroductionSuperalgebra1997} explained below. Worse yet, a Lie superalgebra is not a Lie algebra with grading. 
    
    Still, the prefix super- is popular, and so we provide an overview of the superized definition of rings, modules, and algebras as they are in play in this paper. 
    
    A ring $R = R_{\bar{0}} \oplus R_{\bar{1}}$ with a choice of subgroups $R_{\bar{0}}$ and $R_{\bar{1}}$ such that $R_{\alpha}R_{\beta} \subseteq R_{\alpha + \beta}$, where the indices $\alpha, \beta$ belong to $\mathbb{Z}_{2}$, is a \emph{superring}. For $x = x_{\bar{0}} + x_{\bar{1}} \in R$, let $\overline{\mbox{$x$}} = x_{\bar{0}} - x_{\bar{1}}$. The center $Z(R)$ of a superring $R$ is the subgroup $\{c = c_{\bar{0}} + c_{\bar{1}} \in R \mid cx = xc_{\bar{0}} + \overline{\mbox{$x$}}c_{\bar{1}}, \text{ for all } x = x_{\bar{0}} + x_{\bar{1}} \in R\}$. A superring is \emph{supercommutative} if it equals its center as a superring. 
      
      Fix a superring $R = R_{\bar{0}} \oplus R_{\bar{1}}$. An internal direct sum of abelian groups $M = M_{\bar{0}} \oplus M_{\bar{1}}$ is a \emph{left (respectively, right) $R$-supermodule} if $M$ is a left (respectively, right) $R$-module such that $R_{\alpha}M_{\beta} \subseteq M_{\alpha + \beta}$ (respectively, $M_{\beta}R_{\alpha} \subseteq M_{\alpha + \beta}$) for $\alpha, \beta \in \mathbb{Z}_{2}$.  If $R$ is supercommutative, then the left action of $R$ on $M$ determines uniquely a right $R$-supermodule structure on $M$: Let $x = x_{\bar{0}} + x_{\bar{1}} \in R$, $m = m_{\bar{0}} + m_{\bar{1}} \in M$, with $\overline{\mbox{$m$}} = m_{\bar{0}} - m_{\bar{1}}$, and denote the left action by $xm$. Then the right action is defined by $mx = x_{\bar{0}}m + x_{\bar{1}}\overline{\mbox{$m$}}$. Likewise, a right $R$-supermodule gives a left $R$-supermodule. Thus we call left/right $R$-supermodules plainly \emph{supermodules} when $R$ is supercommutative.

        Now let $B = B_{\bar{0}} \oplus B_{\bar{1}}$ be an $R$-supermodule over a supercommutative ring $R = R_{\bar{0}} \oplus R_{\bar{1}}$. We say $B$ is  a \emph{superalgebra} (over $R$) if $B$ has an $R$-bilinear product satisfying $B_{\alpha} B_{\beta} \subset B_{\alpha + \beta}$ for $\alpha, \beta \in \mathbb{Z}_{2}$. A superalgebra $B$ may be associative, unital (if $B$ has unity), commutative, if it satisfies the usual conditions, or supercommutative using the condition for superrings. 

        Consider a superring/supermodule/superalgebra $S = S_{\bar{0}} \oplus S_{\bar{1}}$. The subgroup $S_{\bar{0}}$ is the even part of $S$ comprising the even elements, and $S_{\bar{1}}$ is the odd part of $S$ comprising the odd elements. Parity refers to the map $x \mapsto |x| = \bar{i}$ for the nonzero homogeneous elements of $S$, which are the nonzero $x \in S_{\bar{i}}$, $\bar{i} \in \mathbb{Z}_{2}$. The zero element is simultaneously even and odd (homogeneous of each parity). The definitions above can be recast using the parity map $|\cdot|$; for example, $R_{\alpha}R_{\beta} \subseteq R_{\alpha + \beta}$ is equivalent to $|xy| = |x| + |y|$ for homogeneous $x,y \in R$, and extending by linearity. Additionally, $Z(R) = \{x \in R \mid xy = (-)^{|x||y|}yx \}$. 
        
    \begin{Remark}[Parity and the rule of signs]
    Returning to Manin's rule of signs, we generally append a factor of $\epsilon(v,w) = (-1)^{|v||w|}$ whenever there is an exchange of elements $v,w$ within a product, as prescribed by the braiding. An application of the rule of signs implies that a supersymmetric form on a super vector space $L = L_{\bar{0}} \oplus L_{\bar{1}}$ is a form on $L$ which is symmetric on $L_{\bar{0}}$ and anti-symmetric on $L_{\bar{1}}$. This is how we proceed in discussing the usual linear algebraic terms through the lens of superalgebra.
    \end{Remark} 
    
We have linear maps which preserve parity (even maps, which are the morphisms of super vector spaces) or reverse parity (odd maps), and the collection of their sums (the set of all linear maps). Note that super vector spaces are preserved under direct sums (the sum of even parts is even; the sum of odd parts is odd) and tensor products (the product of even with even parts or odd with odd parts is even; the product of even with odd parts is odd).

\subsection{Tensor product representations} \label{appendix:TensorPrimitive}
Let $\mathbbm{1}_{S}$ be the identity morphism on a space $S$, or we write $\mathbbm{1}$ plainly with the reader's forgiveness. For any two $\mathfrak{g}$-modules $V$ and $W$, the tensor product $V \otimes W$ carries a natural $\mathfrak{g}$-action:
$X(v \otimes w) = X(v) \otimes w + (-1)^{|x||v|}v \otimes X(w)$. 
Re-expressing the above by highlighting the maps in the representations $(V, \rho), (W, \tau$), we have
\[\rho \otimes \tau\!: \mathfrak{g} \rightarrow \mathfrak{gl}(V \otimes W), \quad x \mapsto \rho(x) \otimes \mathbbm{1} + \mathbbm{1} \otimes \tau(x), \quad \text{for all } x \in \mathfrak{g}, \]
by way of comultiplication. That is to say, $\rho \otimes \tau (v \otimes w)$ does not mean $\rho(v) \otimes (-1)^{|\tau||v|}\tau(w)$ for the tensor product of representations.

\begin{Remark}[Parity considerations]
    For $\mathbbm{1} \otimes A \in \mathfrak{gl}(V \otimes W)$ and  $v \otimes w \in V \otimes W$, \[\big(\mathbbm{1} \otimes A\big)(v\otimes w) = (-1)^{|A||v|} \mathbbm{1}(v) \otimes A(w),\] with $|A|$ defined in the superalgebra $\mathfrak{gl}(W)$.
\end{Remark}

\begin{Example}[{Set $V = \mathbb{C}[\mathbf{x}]$ and $W = \mathbb{C}^{1|2n}$}]\label{example:tensor-product-representation}
	The action is
        \begin{equation}
            X\big(\mathbf{x}^{\mathbf{k}} \otimes v_{j}\big) =  X\big(\mathbf{x}^{\mathbf{k}}\big)\otimes v_{j} + (-1)^{|X||\mathbf{k}|}\mathbf{x}^{\mathbf{k}} \otimes X\big(v_{j}\big)
        \end{equation}
	for each vector $X \in \mathfrak{g}$. In particular,
		\begin{align}
		X_{-\delta_{i}}(\mathbf{x}^{\mathbf{k}} \otimes v_{j}) &= \begin{cases}
		    \frac{1}{\sqrt{2}}k_{i}\mathbf{x}^{\mathbf{k}-\delta_{i}}\otimes v_{0} - (-1)^{|\mathbf{k}|}\mathbf{x}^{\mathbf{k}} \otimes v_{n+i}, \hfill &\text{when } j = 0\\
		    \frac{1}{\sqrt{2}}k_{i}\mathbf{x}^{\mathbf{k}-\delta_{i}}\otimes v_{j} + \delta_{ij}(-1)^{|\mathbf{k}|}\mathbf{x}^{\mathbf{k}} \otimes v_{0}, \hfill &\text{when } 1 \leq j \leq n\\
		    \frac{1}{\sqrt{2}}k_{i}\mathbf{x}^{\mathbf{k}-\delta_{i}}\otimes v_{j}, \hfill &\text{when } n+1 \leq j \leq 2n 
		\end{cases}\\
        X_{\delta_{i}}(\mathbf{x}^{\mathbf{k}} \otimes v_{j}) &=	\begin{cases}
				\frac{1}{\sqrt{2}}\mathbf{x}^{\mathbf{k}+\delta_{i}}\otimes v_{0} + (-1)^{|\mathbf{k}|}\mathbf{x}^{\mathbf{k}} \otimes v_{i}, \hfill &\text{when } j = 0\\
		        \frac{1}{\sqrt{2}}\mathbf{x}^{\mathbf{k}+\delta_{i}}\otimes v_{j}, \hfill &\text{when } 1 \leq j \leq n\\
		        \frac{1}{\sqrt{2}}\mathbf{x}^{\mathbf{k}+\delta_{i}}\otimes v_{j} + \delta_{ij}(-1)^{|\mathbf{k}|}\mathbf{x}^{\mathbf{k}} \otimes v_{0}, \hfill &\text{when } n+1 \leq j \leq 2n.
		\end{cases}
		\end{align}
\end{Example}	

The tensor product of a weight module is again a weight module. Weight vectors are elementary tensors of weight vectors from their respective spaces, where their weights add to determine the weight within the tensor product representation. 

\section{Proof of \texorpdfstring{\Cref{prop:space-of-singular-vectors}}{Proposition 1.1}}\label{appendix:proof}

In $\mathbb{C}[\mathbf{x}]\otimes\mathbb{C}^{1|2n}$,
set
	\begin{align}
	    w^{-\delta}_{1,j} &= 1 \otimes v_{n+j}\label{eq:w1jdelta}\\ 
	    w^{-\delta}_{2} &= 1 \otimes v_{0} + \sqrt{2} \sum\limits_{j=1}^{n} x_{j} \otimes v_{n+j}\label{eq:w2delta}\\
	    w^{-\delta}_{3,j} &= -\sqrt{2} x_{j} \otimes v_{0} + 1 \otimes v_{j} + x_{j}^{2} \otimes v_{n+j}\label{eq:w3jdelta},
	\end{align}
	with $1 \leq j \leq n$, and where the superscript $-\delta$ suggests the following lemma. 

\begin{Lemma}\label{lem:deltaprimitivespace}
    The vectors in the $\mathfrak{g}$-module $\mathbb{C}[\mathbf{x}] \otimes \mathbb{C}^{1|2n}$ annihilated by the positive root vectors $X_{-\delta_{i}}$, $1 \leq i \leq n$, are precisely linear combinations of the vectors $w^{-\delta}_{1,j}$, $w^{-\delta}_{2}$, and $w^{-\delta}_{3,j}$.	
		\end{Lemma}
	\begin{proof}
	Any vector $w \in \mathbb{C}[\mathbf{x}] \otimes \mathbb{C}^{1|2n}$ may be expressed uniquely as a sum
	\begin{equation}\label{eq:vectordecomp}
	w = \sum\limits_{l=1}^{m} \left(\alpha_{0,\lambda_{l}} \mathbf{x}^{\mathbf{k}_{\lambda_{l}}} \otimes v_{0} + \sum\limits_{j=1}^{n} \alpha_{j,\lambda_{l}} \mathbf{x}^{\mathbf{k}_{\lambda_{l}} - \delta_{j}} \otimes v_{j} + \sum\limits_{j=1}^{n} \alpha_{n+j,\lambda_{l}} \mathbf{x}^{\mathbf{k}_{\lambda_{l}}+ \delta_{j}} \otimes v_{n+j}\right), 
	\end{equation}	
	for some positive integer $m$ and \[ \mathbf{k}_{\lambda_{l}} = (k_{1,\lambda_{l}},\thinspace k_{2,\lambda_{l}},\thinspace \ldots,\thinspace k_{n,\lambda_{l}}), \thickspace k_{i,\lambda_{l}} \in \mathbb{Z}_{\geq -1},\thinspace 1 \leq i \leq n.\] 

	\begin{Remark}\label{rmk:Exponents}
    Whenever $\mathbf{x}^{\mathbf{k}_{\lambda_{l}}},\mathbf{x}^{\mathbf{k}_{\lambda_{l}}-\delta_{j}}, \mathbf{x}^{\mathbf{k}_{\lambda_{l}}+\delta_{j}} \not\in \mathbb{C}[\mathbf{x}]$, then  $a_{0,\lambda_{l}}, a_{j, \lambda_{l}}, a_{n+j, \lambda_{l}} = 0$, respectively. We deduce that at most one entry $k_{i,\lambda_{l}}$ in $\mathbf{k}_{\lambda_{l}}$ is negative; indeed, a negative entry would equal $-1$. 
    \end{Remark}
 
 The nonzero summands are $\mathfrak{h}$-weight vectors of weight 
 \[\lambda_{l} = \mathbf{k}_{\lambda_{l}} + \frac{1}{2}\nu_{n}.\] 

	Assume $X_{-\delta_{i}}(w) = 0$ for $1 \leq i \leq n$. Then we have \begin{align*}
		0 &=X_{-\delta_{i}}(w)\\ &= X_{-\delta_{i}}\left(\sum\limits_{l=1}^{m} \left(\alpha_{0,\lambda_{l}} \mathbf{x}^{\mathbf{k}_{\lambda_{l}}} \otimes v_{0} + \sum\limits_{j=1}^{n} \alpha_{j,\lambda_{l}} \mathbf{x}^{\mathbf{k}_{\lambda_{l}} - \delta_{j}} \otimes v_{j} + \sum\limits_{j=1}^{n} \alpha_{n+j,\lambda_{l}} \mathbf{x}^{\mathbf{k}_{\lambda_{l}}+ \delta_{j}} \otimes v_{n+j}\right)\right)\\
		& = \sum\limits_{l=1}^{m}\left( \left(\frac{1}{\sqrt{2}}k_{i,\lambda_{l}}\alpha_{0,\lambda_{l}} + (-1)^{|\mathbf{k}_{\lambda_{l}}-\delta_{i}|}\alpha_{i,\lambda_{l}}\right)\mathbf{x}^{\mathbf{k}_{\lambda_{l}} - \delta_{i}} \otimes v_{0} + \sum\limits_{j=1}^{n} \frac{1}{\sqrt{2}}(k_{i,\lambda_{l}}-\delta_{ij})\alpha_{j,\lambda_{l}} \mathbf{x}^{\mathbf{k}_{\lambda_{l}} - \delta_{j} - \delta_{i}} \otimes v_{j} \right.\\
		& \enskip + \left. \sum\limits_{j=1}^{n}\left(\frac{1}{\sqrt{2}}\left(k_{i,\lambda_{l}}+\delta_{ij}\right)\alpha_{n+j,\lambda_{l}}  -\delta_{ij}(-1)^{|\mathbf{k}_{\lambda_{l}}|} \alpha_{0,\lambda_{l}}\right) \mathbf{x}^{\mathbf{k}_{\lambda_{l}} + \delta_{j}-\delta_{i}} \otimes v_{n+j}\right)
		\end{align*}
We use the standard basis
\begin{equation}\label{standard-tensor-basis}
\{\mathbf{x}^{\mathbf{k}}\otimes v_{0}, \mathbf{x}^{\mathbf{k}} \otimes v_{i} \mid \mathbf{k} \in \mathbb{Z}_{\geq 0}^{n}, 1 \leq i \leq n\}
\end{equation}
of $\mathbb{C}[\mathbf{x}] \otimes \mathbb{C}^{1|2n}$ (and its decomposition into $\mathfrak{h}$-weight spaces) to see the following system of equations characterize which vectors the $X_{-\delta_{i}}$ annihilate by determining conditions on the coefficients and exponents in $\eqref{eq:vectordecomp}$:

\begin{align}
	k_{i,\lambda_{l}}\alpha_{0, \lambda_{l}} &= (-1)^{|\mathbf{k}_{\lambda_{l}}|}\sqrt{2}\alpha_{i,\lambda_{l}}, & \hfill \text{for } 1 \leq i \leq n \label{eq:vnoughtinfo}\\
	(k_{i,\lambda_{l}}-\delta_{ij})\alpha_{j,\lambda_{l}} &= 0, & \hfill \text{for } 1 \leq i \leq n \label{eq:viinfo}\\ 
	(k_{i,\lambda_{l}}+\delta_{ij})\alpha_{n+j,\lambda_{l}} &= \delta_{ij}(-1)^{|\mathbf{k}_{\lambda_{l}}|}\sqrt{2} \alpha_{0,\lambda_{l}},  & \hfill \text{for } 1 \leq i \leq n. \label{eq:vnplusiinfo}
\end{align}

By \eqref{eq:viinfo}, at most one $\alpha_{j,\lambda_{l}}$ may be nonzero. Let $\alpha_{p,\lambda_{l}} \in \mathbb{C}^{\ast}$ so that $\alpha_{p,\lambda_{l}}\mathbf{x}^{\mathbf{k}-\delta_{p}} \otimes v_{p}$ is a unique nonzero term of weight $\lambda_{l}$.
Then
\begin{align}
k_{j,\lambda_{l}} &= \delta_{jp}, & \hfill \text{for } 1 \leq j \leq n  \label{eq:klambdaequaldeltajp}\\
\alpha_{0,\lambda_{l}} &= -\sqrt{2}\alpha_{p,\lambda_{l}} \label{eq:a0equalnroottwoap}\\
\alpha_{j,\lambda_{l}} = \alpha_{n+j,\lambda_{l}} &= \delta_{jp}, & \hfill \text{for } 1 \leq j \leq n.\label{eq:ajequalajplusnqualdeltajp}
\end{align}
Equations \eqref{eq:klambdaequaldeltajp} - \eqref{eq:ajequalajplusnqualdeltajp} imply $X_{-\delta_{i}}(\mathbb{C}w_{3,j}) = 0$ for each $i$ of the first $n$ integers, recalling \eqref{eq:w3jdelta}.

In the case that all $\alpha_{j, \lambda_{l}}$ are zero, then either $\alpha_{0,\lambda_{l}} = 0$ or $\mathbf{k}_{\lambda_{l}} = \mathbf{0}$. 
The former, $\alpha_{0,\lambda_{l}} = 0$, permits at most one nonzero $\alpha_{n+j,\lambda_{l}}$ in the $\lambda_{l}$ summand of \eqref{eq:vectordecomp}. Let $\alpha_{n+q,\lambda_{l}} \in \mathbb{C}^{\ast}$ so that $\alpha_{n+q,\lambda_{l}}\mathbf{x}^{\mathbf{k}+\delta_{q}} \otimes v_{n+q}$ is the unique nonzero term of weight $\lambda_{l}$. Then

\begin{align}
k_{j,\lambda_{l}} &= -\delta_{jq}, & \hfill \text{for } 1 \leq j \leq n  \label{eq:klambdaequalndeltaq}\\
\alpha_{0,\lambda_{l}} &= 0 \label{eq:a0equal0}\\
\alpha_{j,\lambda_{l}} &= 0, & \hfill \text{for } 1 \leq j \leq n\label{eq:ajequal0}\\ \alpha_{n+j,\lambda_{l}} &= \delta_{jq}\alpha_{n+q,\lambda_{l}}, & \hfill \text{for } 1 \leq j \leq n.\label{eq:ajplusnequaldeltaaq}
\end{align}

Equations \eqref{eq:klambdaequalndeltaq} - \eqref{eq:ajplusnequaldeltaaq} imply $X_{-\delta_{i}}(\mathbb{C}w_{1,j}) = 0$ for each $i$ of the first $n$ integers, recalling \eqref{eq:w1jdelta}.

While considering distinctly the latter, $\mathbf{k} = \mathbf{0}$, gives $a_{0,\lambda_{l}}$ as a nonzero scalar. Then
\begin{align}
k_{j,\lambda_{l}} &= 0, & \hfill \text{for } 1 \leq j \leq n  \label{eq:klambdaequal0}\\
\alpha_{0,\lambda_{l}} &\in \mathbb{C}\label{eq:a0nonzero}\\
\alpha_{j,\lambda_{l}} &= 0, & \hfill \text{for } 1 \leq j \leq n\label{eq:ajequal0again}\\ \alpha_{n+j,\lambda_{l}} &= \sqrt{2}\alpha_{0,\lambda_{l}}, & \hfill \text{for } 1 \leq j \leq n,\label{eq:ajplusnequalroottwoa0}
\end{align}

Equations \eqref{eq:klambdaequal0} - \eqref{eq:ajplusnequalroottwoa0} imply $X_{-\delta_{i}}(\mathbb{C}w_{2}) = 0$ for each $i$ of the first $n$ integers, recalling \eqref{eq:w2delta}.

Thus, if $w$ is annihilated by the positive root vectors $X_{-\delta_{i}}$, $1 \leq i \leq n$, then
\[w \in \sum\limits_{j=1}^{n}\mathbb{C}w^{-\delta}_{1,j} + \mathbb{C}w^{-\delta}_{2} +
    \sum\limits_{j=1}^{n} \mathbb{C}w^{-\delta}_{3,j}.\]
\end{proof}
Call $\sum\limits_{j=1}^{n}\mathbb{C}w^{-\delta}_{1,j} + \mathbb{C}w^{-\delta}_{2} +
\sum\limits_{j=1}^{n} \mathbb{C}w^{-\delta}_{3,j}$ the ${-\delta}$-space, which means the space of vectors annihilated by the span of root vectors associated to the roots $-\delta_{i}$ and denote the space by $\mathbb{C}[\mathbf{x}]\otimes\mathbb{C}^{1|2n}[-\delta]$. 
Now for $1 \leq j \leq n$, consider the vectors 

	\begin{align}
	    w^{\delta - \delta}_{v_{0}, k_{1}} &= x_{1}^{k_{1}} \otimes v_{0} \label{eq:w0deltadelta}\\
	    w^{\delta - \delta}_{k_{1}} &= x_{1}^{k_{1}}\sum\limits_{j=1}^{n} x_{j} \otimes v_{n+j}\label{eq:wk1deltadelta}\\
	    w^{\delta - \delta}_{v_{2n},k_{1}} &= x_{1}^{k_{1}} \otimes v_{2n}\label{eq:w2ndeltadelta}\\
	    w^{\delta - \delta}_{v_{1},k_{1}} &= x_{1}^{k_{1}} \otimes v_{1}\label{eq:w1deltadelta}\\
	    w^{\delta - \delta}_{k_{1},x_{2}} &=
	    x_{1}^{k_{1}}x_{2}^{1} \otimes v_{1} - x_{1}^{k_{1}} \otimes v_{2} \label{eq:wx2deltadelta}.
	\end{align}

	We have the following lemma concerning the $\delta-\delta$-space $\mathbb{C}[\mathbf{x}] \otimes \mathbb{C}^{1|2n}[\delta - \delta]$, the space of vectors annihilated by $X_{\delta_{i} - \delta_{i+1}}$, $1 \leq i \leq n-1$.

\begin{Lemma}\label{lem:deltadeltaprimitivespace}
    The $(\delta-\delta)$-space $\mathbb{C}[\mathbf{x}] \otimes \mathbb{C}^{1|2n}[\delta - \delta]$ is spanned by the vectors of  \eqref{eq:w0deltadelta} - \eqref{eq:wx2deltadelta}.
		\end{Lemma}
	\begin{proof}
	Again, any vector $w \in \mathbb{C}[\mathbf{x}] \otimes \mathbb{C}^{1|2n}$ may be expressed uniquely as a sum
	\begin{equation}\tag{\ref{eq:vectordecomp} recalled} 
	w = \sum\limits_{l=1}^{m} \left(\alpha_{0,\lambda_{l}} \mathbf{x}^{\mathbf{k}_{\lambda_{l}}} \otimes v_{0} + \sum\limits_{j=1}^{n} \alpha_{j,\lambda_{l}} \mathbf{x}^{\mathbf{k}_{\lambda_{l}} - \delta_{j}} \otimes v_{j} + \sum\limits_{j=1}^{n} \alpha_{n+j,\lambda_{l}} \mathbf{x}^{\mathbf{k}_{\lambda_{l}}+ \delta_{j}} \otimes v_{n+j}\right),
	\end{equation}	
	for some positive integer $m$, where the summands are $\mathfrak{h}$-weight vectors of weight \[\lambda_{l} = \mathbf{k}_{\lambda_{l}} + \frac{1}{2}\nu_{n},\quad k_{i,\lambda_{l}} \in \mathbb{Z}_{\geq -1},\thinspace 1 \leq i \leq n.\] 
	
	Recall Remark \ref{rmk:Exponents} and assume $X_{\delta_{i} - \delta_{i+1}}(w) = 0$ for $1 \leq i \leq n$. Then we have \begin{align*}
		0 &=X_{\delta_{i} - \delta_{i+1}}(w)\\ &= X_{\delta_{i} - \delta_{i+1}}\left(\sum\limits_{l=1}^{m} \left(\alpha_{0,\lambda_{l}} \mathbf{x}^{\mathbf{k}_{\lambda_{l}}} \otimes v_{0} + \sum\limits_{j=1}^{n} \alpha_{j,\lambda_{l}} \mathbf{x}^{\mathbf{k}_{\lambda_{l}} - \delta_{j}} \otimes v_{j} + \sum\limits_{j=1}^{n} \alpha_{n+j,\lambda_{l}} \mathbf{x}^{\mathbf{k}_{\lambda_{l}}+ \delta_{j}} \otimes v_{n+j}\right)\right)\\
		& = \sum\limits_{l=1}^{m}\left( k_{i+1,\lambda_{l}}\alpha_{0,\lambda_{l}} \mathbf{x}^{\mathbf{k}_{\lambda_{l}}+\delta_{i}-\delta_{i+1}} \otimes v_{0} + \sum\limits_{j=1}^{n} \Big((k_{i+1,\lambda_{l}}-\delta_{i+1,j})\alpha_{j,\lambda_{l}} + \delta_{ij}\alpha_{j+1, \lambda_{l}} \Big)\mathbf{x}^{\mathbf{k}_{\lambda_{l}} - \delta_{j} + \delta_{i} - \delta_{i+1}} \otimes v_{j} \right.\\
	   & \enskip + \left. \sum\limits_{j=1}^{n} \Big((k_{i+1,\lambda_{l}}+\delta_{i+1,j})\alpha_{n+j,\lambda_{l}} - \delta_{i+1,j}\alpha_{n+j-1, \lambda_{l}} \Big)\mathbf{x}^{\mathbf{k}_{\lambda_{l}} + \delta_{j} + \delta_{i} - \delta_{i+1}} \otimes v_{n+j}    \right)
 		\end{align*}
	As before, we get a system of equations:
\begin{align}
	k_{i+1,\lambda_{l}}\alpha_{0, \lambda_{l}} &= 0, & \hfill \text{for } 1 \leq i \leq n-1 \label{eq:vnoughtinfo2}\\
	(k_{i+1,\lambda_{l}}-\delta_{i+1,j})\alpha_{j,\lambda_{l}} &= -\delta_{ij}\alpha_{j+1,\lambda_{l}}, & \hfill \text{for } 1 \leq i \leq n-1 \label{eq:viinfo2}\\ 
	(k_{i+1,\lambda_{l}}+\delta_{i+1,j})\alpha_{n+j,\lambda_{l}} &= \delta_{i+1,j}\alpha_{n+j-1, \lambda_{l}},  & \hfill \text{for } 1 \leq i \leq n-1. \label{eq:vnplusiinfo2}
\end{align}

One can check that the above system of equations gives rise to the following cases:

\begin{enumerate}[label=Case \arabic*, leftmargin=2cm]
     \item Let $a_{n+1} = 0$.
         \begin{enumerate}[label=Subcase (\arabic{enumi}.\roman*),leftmargin=2.25cm] 
         \item Then either $w = 0$ or \label{case:1i}
         \item $a_{n+2} = a_{n+3} = \cdots = a_{2n-1} = 0$ and $k_{2} = k_{3} = \cdots = k_{n-1} = 0$, with $a_{2n} \neq 0$ and $k_{n} = -1$, hence $a_{0} = 0$. \label{case:1ii}
         \end{enumerate}
     \item Let $a_{n+1} \neq 0$.
         \begin{enumerate}[label=Subcase (\arabic{enumi}), leftmargin=2.25cm]
         \item[] Then $a_{n+1} = a_{n+2} = \cdots = a_{2n} \neq 0$ \& $k_{2} = k_{3} = \cdots = k_{n} = 0$. \label{case:2i}
         \end{enumerate}
     \item Let $a_{1} = 0$.
         \begin{enumerate}[label=Subcase (\arabic{enumi}), leftmargin=2.25cm]
         \item[] Then $a_{2} = a_{3} = \cdots = a_{n} = 0$. \label{case:3i}
         \end{enumerate}
     \item Let $a_{1} \neq 0$.
         \begin{enumerate}[label=Subcase (\arabic{enumi}.\roman*), leftmargin=2.25cm]
         \item Then $a_{2} = a_{3} = \cdots = a_{n} = 0$ \text{ and } $k_{2} = k_{3} = \cdots = k_{n} = 0$ \label{case:4i} or
         \item $a_{3} = a_{4} = \cdots = a_{n} = 0$ \text{ and } $k_{3} = k_{4} = \cdots = k_{n} = 0$, with $a_{2} = -a_{1}$ \text{ and } $k_{2} = 1$, hence $a_{0} = 0$ \label{case:4ii}.
         \end{enumerate}
\end{enumerate}
\begin{enumerate}

\item[] Regard \ref{case:1i} as the necessary case of $w = 0 \in \mathbb{C}[\mathbf{x}] \otimes \mathbb{C}^{1|2n}$.

\item[] From \ref{case:1ii}: We have  $X_{\delta_{i} - \delta_{i+1}}(\mathbb{C}w^{\delta - \delta}_{v_{2n},k_{1}}) = 0$ for $1 \leq i \leq n-1$. 

\item[] From \ref{case:2i}: We have  $X_{\delta_{i} - \delta_{i+1}}(\mathbb{C}w^{\delta - \delta}_{v_{0}, k_{1}} + \mathbb{C}w^{\delta - \delta}_{k_{1}}) = 0$ for $1 \leq i \leq n-1$.   

\item[] From \ref{case:4i}: We have $X_{\delta_{i} - \delta_{i+1}}(\mathbb{C}w^{\delta - \delta}_{v_{0}, k_{1}} + \mathbb{C}w^{\delta - \delta}_{v_{1},k_{1}}) = 0$ for $1 \leq i \leq n-1$.

\item[] From \ref{case:4ii}: We have $X_{\delta_{i} - \delta_{i+1}}(\mathbb{C}w^{\delta - \delta}_{k_{1}, x_{2}}) = 0$ for $1 \leq i \leq n-1$
\end{enumerate}
\ref{case:3i} implies one of the previous outcomes. 

We have shown that $w \in \mathbb{C}[\mathbf{x}] \otimes \mathbb{C}^{1|2n}$ is annihilated by $X_{\delta_{i} - \delta_{i+1}}$, with $1 \leq n-1$, if and only if $w$ is in the span of \[\{x_{1}^{k_{1}} \otimes v_{0},\, x_{1}^{k_{1}}\sum\limits_{j=1}^{n} x_{j} \otimes v_{n+j},\, x_{1}^{k_{1}} \otimes v_{2n},\, x_{1}^{k_{1}} \otimes v_{1},\, x_{1}^{k_{1}}x_{2}^{1} \otimes v_{1} - x_{1}^{k_{1}} \otimes v_{2} \}.\]
\end{proof}

Now we prove \Cref{prop:space-of-singular-vectors}.

\begin{proof}[Proof of \Cref{prop:space-of-singular-vectors}]
For $n \in \mathbb{Z}_{>0}$, let $V_{n} = \mathbb{C}[\mathbf{x}] \otimes \mathbb{C}^{1|2n}$. Denote by $V_{n}^{+}$ the subspace of singular vectors in the $\mathfrak{g}$-module $V_{n}$. Note that $V_{n}^{+} = V_{n}[-\delta] \cap V_{n}[\delta-\delta]$, which simplifies to $V_{1}^{+} = V_{1}[-\delta]$ when $n = 1$. 
Thus when $n > 1$, we appeal to Lemmas \ref{lem:deltaprimitivespace} and \ref{lem:deltadeltaprimitivespace}. For each $n > 1$, we have  
\begin{align}\label{eq:n>1prim}
w_{1} &\coloneqq w^{\delta}_{1,n} = w^{\delta - \delta}_{2n,0}\\
w_{2} &\coloneqq w^{\delta}_{2} = w^{\delta - \delta}_{v_{0},1} + \sqrt{2} w^{\delta - \delta}_{0}.
\end{align}
Recall the standard basis of $V_{n}$ (\Cref{standard-tensor-basis}). Then 
$V_{n}^{+} = V_{n}[-\delta] \cap V_{n}[\delta-\delta] = \mathbb{C}w^{1} \oplus \mathbb{C}w_{2}$
is a $2$-dimensional subspace of $V_{n}$ for $n > 1$. Otherwise, when $n = 1$, we appeal solely to Lemma \ref{lem:deltaprimitivespace}, and  Equations \eqref{eq:w1jdelta} - \eqref{eq:w3jdelta} give us 
\begin{align*}\label{eq:n=1prim}
w_{1} &\coloneqq w^{\delta}_{1,1}, \quad w_{2} \coloneqq w^{\delta}_{2}, \quad w_{3} \coloneqq w^{\delta}_{3,1}.
\end{align*}
Thus $V_{1}^{+} = V_{1}[-\delta] = \mathbb{C}w_{1} \oplus \mathbb{C}w_{2} \oplus \mathbb{C}w_{3}$ is a $3$-dimensional subspace of $V_{1}$. 
\end{proof}

	\printbibliography
	
\end{document}